\documentclass{siamart1116}

\usepackage{lipsum}
\usepackage{amsfonts}
\usepackage{graphicx}
\usepackage{epstopdf}
\usepackage{algorithmic}
\ifpdf
  \DeclareGraphicsExtensions{.eps,.pdf,.png,.jpg}
\else
  \DeclareGraphicsExtensions{.eps}
\fi

\numberwithin{theorem}{section}

\newcommand{\TheTitle}{Monotonicity in inverse medium scattering on unbounded domains}
\newcommand{\RunningTitle}{Monotonicity in inverse medium scattering}  
\newcommand{\TheAuthors}{R. Griesmaier and B. Harrach} 

\headers{\RunningTitle}{\TheAuthors}

\title{\TheTitle}

\author{
  Roland Griesmaier\thanks{Institut f\"ur Angewandte und Numerische
    Mathematik, Karlsruher Institut f\"ur Technologie, 76049 Karlsruhe, Germany
    (\email{roland.griesmaier@kit.edu}).} 
  \and
  Bastian Harrach\thanks{Institut f\"ur Mathematik
    Universit\"at Frankfurt, 60325 Frankfurt am Main, Germany
    (\email{harrach@math.uni-frankfurt.de}).}
}

\usepackage{amsmath,amssymb}
\usepackage{mathrsfs} 
\usepackage{bm}
\usepackage{stmaryrd}

\theorembodyfont{\normalfont}

\newtheorem{example}[theorem]{Example}
\newtheorem{remark}[theorem]{Remark}

\numberwithin{equation}{section}
\numberwithin{figure}{section}
\numberwithin{theorem}{section}

\theoremstyle{nonumberplain}
\theoremheaderfont{\normalfont\itshape}
\theorembodyfont{\normalfont}
\theoremseparator{.}
\theoremsymbol{}

\newcommand{\Lt}{{L^2}}
\newcommand{\Linfty}{{L^{\infty}}}
\newcommand{\LinftyC}{{L^\infty_{0,+}}}
\newcommand{\Heins}{{H^1}}

\newcommand{\LtSd}{{L^2(\Sd)}}
\newcommand{\LinftyCRd}{{\LinftyC(\Rd)}}
\newcommand{\Heinsloc}{{H^1_\loc(\Rd)}}

\newcommand{\field}[1]{{\mathbb{#1}}}
\newcommand{\C}{\field{C}}
 
\newcommand{\N}{\field{N}}
\newcommand{\R}{\field{R}}

\newcommand{\Dcal}{\mathcal{D}}

\newcommand{\Rcal}{\mathcal{R}}
\newcommand{\Scal}{\mathcal{S}}

\newcommand{\bs}{\boldsymbol}

\newcommand{\bfA}{{\bs A}}

\newcommand{\bfF}{{\bs F}}

\newcommand{\bfI}{{\bs I}}

\newcommand{\bfS}{{\bs S}}
\newcommand{\bfT}{{\bs T}}

\newcommand{\loc}{{\mathrm{loc}}}

\newcommand{\ol}[1]{\overline{#1}}

\newcommand{\tm}{\subseteq} 
\newcommand{\di}{\partial}

\newcommand{\ds}{\, \dif s}

\newcommand{\dx}{\, \dif x}
\newcommand{\dy}{\, \dif y}

\newcommand{\xhat}{\widehat{x}}
\newcommand{\yhat}{\widehat{y}}

\newcommand{\gtilde}{{\widetilde g}}

\newcommand{\Vtilde}{{\widetilde V}}

\newcommand{\rmi}{\mathrm{i}}

\newcommand{\eps}{\varepsilon}

\newcommand{\uinfty}{u^\infty}
\newcommand{\ui}{u^i}
\newcommand{\us}{u^s}

\newcommand{\Sd}{{S^{d-1}}}
\newcommand{\Rd}{{\R^d}}

\newcommand{\qmin}{q_\mathrm{min}}
\newcommand{\qmax}{q_\mathrm{max}}
\newcommand{\qminE}{q_{\mathrm{min},E}}
\newcommand{\qmaxE}{q_{\mathrm{max},E}}

\newcommand{\BR}{{B_R(0)}}

\newcommand{\leqfin}{\leq_{\mathrm{fin}}}
\newcommand{\geqfin}{\geq_{\mathrm{fin}}}
\newcommand{\leqd}{\leq_{r}}
\newcommand{\geqd}{\geq_{r}}

\newcommand{\Vperp}{{V^\perp}}

\DeclareMathAlphabet{\mathbi}{\encodingdefault}{\rmdefault}{\bfdefault}{\itdefault}
\DeclareRobustCommand{\vec}[1]{\ifmmode\mathbi{#1}\else\textbf{\textit{#1}}\fi}

\DeclareMathOperator{\dif}{d\!}

\DeclareMathOperator{\supp}{supp}

\DeclareMathOperator{\real}{Re}

\DeclareMathOperator{\sign}{sign}

\DeclareMathOperator{\sinc}{sinc}

\usepackage[resetlabels]{multibib}
\newcites{erratum}{References}
\renewcommand*{\thefootnote}{\fnsymbol{footnote}}

\ifpdf
\hypersetup{
  pdftitle={\TheTitle},
  pdfauthor={\TheAuthors}
}
\fi

\begin{document}

\maketitle

\let\thefootnote\relax\footnotetext{\hrule \vspace{1ex} \centering This is a preprint version of a journal article published in\\
 \emph{SIAM J. Appl. Math.} \textbf{78}(5), 2533--2557, 2018
(\url{https://doi.org/10.1137/18M1171679}).
}

\begin{abstract}
  We discuss a time-harmonic inverse scattering problem for the
  Helmholtz equation with compactly supported penetrable and possibly
  inhomogeneous scattering objects in an unbounded homogeneous
  background medium, and we develop a monotonicity relation for the far
  field operator that maps superpositions of incident plane waves to
  the far field patterns of the corresponding scattered waves. 
  We utilize this monotonicity relation to establish novel 
  characterizations of the support of the scattering objects in terms
  of the far field operator.
  These are related to and extend corresponding results known from
  factorization and linear sampling methods to determine the support
  of unknown scattering objects from far field observations of
  scattered fields.
  An attraction of the new characterizations is that they only
  require the refractive index of the scattering objects to be above
  or below the refractive index of the background medium locally and
  near the boundary of the scatterers. 
  An important tool to prove these results are so-called localized
  wave functions that have arbitrarily large norm in some prescribed
  region while at the same time having arbitrarily small norm in some
  other prescribed region. 
  We present numerical examples to illustrate our theoretical
  findings. 
\end{abstract}



\section{Introduction}
\label{sec:Introduction}
Accurately recovering the location and the shape of unknown scattering
objects from far field observations of scattered acoustic or
electromagnetic waves is a basic but severely ill-posed inverse
problem in remote sensing, and in the past twenty years 
efficient qualitative reconstruction methods for this purpose have
received a lot of attention (see, e.g., 
\cite{cakoni2014qualitative,CakColMon11,colton2013inverse,kirsch2008factorization,potthast2006survey}
and the references therein). 
In this work we develop a new approach for this shape
reconstruction problem that is based on a \emph{monotonicity relation}
for the far field operator that maps superpositions of incident plane
waves, which are being scattered at the unknown scattering objects,
to the far field patterns of the corresponding scattered waves.  
Throughout we assume that the scattering objects are penetrable,
non-absorbing, and possibly inhomogeneous.

The new monotonicity relation generalizes similar results for the
Neumann-to-Dirichlet map for the Laplace equation on bounded domains
that have been established in~\cite{harrach2013monotonicity}, where
they have been utilized to justify and extend an earlier monotonicity
based reconstruction scheme for electrical impedance tomography
developed in~\cite{tamburrino2002new}, using so-called localized
potentials introduced in \cite{gebauer2008localized}. 
This is also related to corresponding estimates for the Laplace equation 
developed in \cite{ikehata1998size,kang1997inverse}. 
The analysis from \cite{harrach2013monotonicity} has recently been
extended for the Neumann-to-Dirichlet operator for the Helmholtz
equation on bounded domains in \cite{harrach2017monotonicity}, and the
main contribution of the present work is the generalization of these 
results to the inverse medium scattering problem on unbounded domains
with plane wave incident fields and far field observations of the
scattered waves. 

The monotonicity relation for the far field operator essentially
states that the real part of a suitable unitary transform of the
difference of two far field operators corresponding to two different
inhomogeneous media is positive or negative semi-definite up to a
finite dimensional subspace, if the difference of the corresponding
refractive indices is either non-negative or non-positive pointwise
almost everywhere. 
This can be translated into criteria and algorithms for shape
reconstruction by comparing a given (or observed) far field operator
to various  virtual (or simulated) far field operators corresponding to
a sufficiently small or large index of refraction on some probing
domains to decide whether these probing domains are contained inside
the support of the unknown scattering objects or whether the probing
domains contain the unknown scattering objects.
In fact the situation is even more favorable, since it turns out to
be sufficient to compare the given far field operator to linearized
versions of the probing far field operators, i.e., Born far field
operators, which can be simulated numerically very efficiently. 
An advantage of these new characterizations is that they only require
the refractive index of the scattering object to be above or below the
refractive index of the background medium locally and near the
boundary of the scatterers, i.e., they apply to a large class of
so-called indefinite scatterers.

Besides the monotonicity relation, the
second main ingredient of our analysis are so-called \emph{localized
  wave functions}, which are special solutions to scattering problems
corresponding to suitably chosen incident waves that have arbitrarily
large norm on some prescribed region $B\tm\Rd$, while at the same time
having arbitrarily small norm on a different prescribed region $D\tm\Rd$,
assuming that $\Rd\setminus\ol{D}$ is connected and $B\not\tm D$. 
This generalizes corresponding results on so-called localized
potentials for the Laplace equation established in
\cite{gebauer2008localized}. 
The arguments that we use to prove the existence of such localized
wave functions are inspired by the analysis of the factorization
method (see \cite{Bru01,Kir98,kirsch1999factorization,kirsch2002music}
for the origins of the method and
\cite{hanke2011sampling,harrach2013recent,kirsch2008factorization} for
recent overviews),  
and of the linear sampling method for the inverse medium scattering
problem (see,
e.g.,~\cite{cakoni2014qualitative,CakColMon11,colton1996simple}).  

It is interesting to note that 
the characterizations of the support of the scattering objects in
terms of the far field operator developed in this work are independent
of so-called transmission eigenvalues (see, e.g.,
\cite{cakoni2014qualitative,CakColHad16,colton2013inverse} and
\cite{Lec09}). 
On the other hand, the monotonicity relation for the far field
operator is somewhat related to well-known monotonicity principles for
the phases of the eigenvalues of the so-called scattering operator,
which have been discussed, e.g., in \cite{kirsch2013inside}, where
they have actually been utilized to characterize transmission
eigenvalues. 
The latter have recently been extended to monotonicity relations for
the difference of far field operators in
\cite{lakshtanov2016difference} that are closely related to our
results.
Our work substantially extends the results in
\cite{lakshtanov2016difference}, using very different analytical
tools. 

For further recent contributions on monotonicity based reconstruction
methods for various inverse problems for partial differential equations we
refer to
\cite{barth2017detecting,brander2017monotonicity,garde2017comparison,garde2017convergence,garde2017regularized,harrach2015combining,harrach2017monotonicity_fractional,harrach2016enhancing,harrach2016monotonicity,harrach2015resolution,maffucci2016novel,su2017monotonicity,tamburrino2016monotonicity,ventre2017design,zhou2017monotonicity}.
We further note that this approach has also been utilized to obtain
theoretical uniqueness results for inverse problems (see, e.g., 
\cite{arnold2013unique,harrach2009uniqueness,harrach2012simultaneous,harrach2010exact,harrach2017local}). 

The outline of this article is as follows.  
After briefly introducing the mathematical setting of the scattering
problem in Section~\ref{sec:DirectProblem}, we develop the
monotonicity relation for the far field operator in
Section~\ref{sec:Monotonicity}.  
In Section~\ref{sec:Localized} we discuss the existence of localized
wave functions for the Helmholtz equation in unbounded domains, and we
use them to provide a converse of the monotonicity relation from
Section~\ref{sec:Monotonicity}.
In Section~\ref{sec:ShapeReconstruction} we establish rigorous
characterizations of the support of scattering objects in terms of
the far field operator.
An efficient and suitably regularized numerical implementation of these
criteria is beyond the scope of this article, but we discuss a
preliminary algorithm and two numerical examples for the sign-definite
case (i.e., when the refractive index of the scattering objects is either
above or below the refractive index of the background medium) in
Section~\ref{sec:NumericalExamples} to illustrate our theoretical
findings.
This preliminary algorithm cannot be considered competitive when
compared against state-of-the-art implementations of linear sampling
or factorization methods, but, as outlined in our final remarks, this
may change in the future.

\section{Scattering by an inhomogeneous medium}
\label{sec:DirectProblem}
We use the Helmholtz equation as a simple model for the propagation of
time-harmonic acoustic or electromagnetic waves in an isotropic
non-absorbing inhomogeneous medium in $\Rd$, $d=2,3$.
Assuming that the inhomogeneity is compactly supported, the 
\emph{refractive index} can be written as  $n^2=1+q$ with a
real-valued \emph{contrast function} $q\in\LinftyCRd$, where 
$\LinftyCRd$ denotes the space of compactly supported
$L^\infty$-functions satisfying $q>-1$ a.e.\ on $\Rd$.  

The wave motion caused by an \emph{incident field} $\ui$
satisfying 
\begin{equation}
  \label{eq:HelmholtzUi}
  \Delta \ui + k^2 \ui \,=\, 0 \qquad \text{in } \Rd, 
\end{equation}
with \emph{wave number} $k>0$, that is being scattered at the
inhomogeneous medium is described by the \emph{total field} $u_q$,
which is a superposition 
\begin{subequations}
  \label{eq:ScatteringProblem}  
  \begin{equation}
    \label{eq:ScatteredField}
    u_q \,=\, \ui+\us_q
  \end{equation}
  of the incident field and the \emph{scattered field} $\us_q$ such
  that the Helmholtz equation 
  \begin{equation}
    \label{eq:Helmholtz}
    \Delta u_q + k^2n^2 u_q \,=\, 0 \qquad \text{in } \Rd \,
  \end{equation}
  is satisfied together with the
  \emph{Sommerfeld radiation condition} 
  \begin{equation}
    \label{eq:Sommerfeld}
    \lim_{r\to\infty} r^{\frac{d-1}2} 
    \Bigl(\frac{\di \us_q}{\di r}(x)-\rmi k \us_q(x)\Bigr) 
    \,=\, 0 \,, \qquad r=|x| \,,
  \end{equation}
\end{subequations}
uniformly with respect to all directions $x/|x|\in\Sd$.

\begin{remark}
  Throughout this work, Helmholtz equations are always to be
  understood in distributional (or weak) sense. 
  For instance, $u_q\in\Heinsloc$ is a solution to
  \eqref{eq:Helmholtz} if and only if  
  \begin{equation*}
    \int_\Rd (\nabla u_q \cdot\nabla v - k^2 n^2 u_q v) \dx \,=\, 0
    \qquad \text{for all $v\in C_0^\infty(\Rd)$\,.} 
  \end{equation*}
  Accordingly, standard regularity results yield smoothness of $u_q$
  and $\us_q$ in $\Rd\setminus\ol{\BR}$, where $\BR$ is a ball
  containing the support of the contrast function $\supp(q)$, and the
  entire solution $\ui$ is smooth throughout $\Rd$. 
  In particular the Sommerfeld radiation condition
  \eqref{eq:Sommerfeld} is well defined.\footnote{As usual, we call
    a (weak) solution to a Helmholtz equation on an unbounded domain
    that satisfies the Sommerfeld radiation condition a
    \emph{radiating solution}.}
  \hfill$\lozenge$
\end{remark}

\begin{lemma}
  Suppose that the incident field $\ui\in\Heinsloc$ satisfies
  \eqref{eq:HelmholtzUi}, then the scattering problem
  \eqref{eq:ScatteringProblem} has a unique solution $u_q\in\Heinsloc$. 
  Furthermore, the scattered field $\us_q=u_q-\ui\in\Heinsloc$ has the
  asymptotic behavior 
  \begin{equation}
    \label{eq:FarfieldExpansion}
    \us_q(x) = C_d \frac{e^{\rmi k |x|}}{|x|^{\frac{d-1}{2}}} 
    \uinfty_q(\xhat) + O(|x|^{-\frac{d+1}{2}}) \,, \qquad 
    |x| \to \infty \,,
  \end{equation}
  uniformly in all directions $\xhat:=x/|x|\in\Sd$,
  where
  \begin{equation}
    \label{eq:DefCd}
    C_d \,=\, {e^{\rmi\pi/4}}/{\sqrt{8\pi k}}
    \quad \text{if $n=2$} \qquad \text{and} \qquad
    C_d \,=\, 1/(4\pi) \quad \text{if $n=3$}\,,
  \end{equation}
  and the \emph{far field pattern} $\uinfty_q$ is given by
  \begin{equation}
    \label{eq:FarfieldPattern}
    \uinfty_q(\xhat) 
    \,=\, \int_{\di B_R(0)} \Bigl( \us_q(y) 
    \frac{\di e^{-\rmi k \xhat\cdot y}}{\di\nu_y} 
    - e^{-\rmi k \xhat\cdot y} 
    \frac{\di \us_q}{\di\nu}(y) \Bigr) \ds(y) \,,
    \qquad \xhat\in\Sd \,.
  \end{equation}
\end{lemma}

\begin{proof}
  The unique solvability follows, e.g., immediately
  from~\cite[Thm.~8.7]{colton2013inverse} (see also 
  \cite[Thm.~6.9]{kirsch2011introduction}), and the farfield
  asymptotics are, e.g., shown in~\cite[Thm.~2.6]{colton2013inverse}. 
\end{proof}

For the special case of a \emph{plane wave incident field} 
$\ui(x;\theta):=e^{\rmi k\theta\cdot x}$, we explicitly indicate the
dependence on the \emph{incident direction} $\theta\in\Sd$ by a second 
argument, and accordingly we write 
$u_q(\cdot;\theta)$, $\us_q(\cdot;\theta)$, and 
$\uinfty_q(\cdot;\theta)$ for the corresponding scattered field, total
field, and far field pattern, respectively. 
As usual, we collect the far field patterns $\uinfty_q(\xhat;\theta)$
for all possible observation and incident directions
$\xhat,\theta\in\Sd$ in the \emph{far field operator}
\begin{equation}
  \label{eq:FarfieldOperator}
  F_q: \LtSd \to \LtSd \,,\quad
  (F_qg)(\xhat)
  := \int_{\Sd} \uinfty_q(\xhat;\theta) g(\theta) \ds(\theta) \,,
\end{equation}
which is compact and normal (see, e.g.,
\cite[Thm.~3.24]{colton2013inverse}). 
Moreover, the \emph{scattering operator} is defined by
\begin{equation}
  \label{eq:ScatteringOperator}
  \Scal_q: \LtSd \to \LtSd \,,\quad
  \Scal_q g := (I+2\rmi k|C_d|^2F_q)g \,,
\end{equation}
where $C_d$ is again the constant from \eqref{eq:DefCd}.
The operator $\Scal_q$ is unitary, and consequently the eigenvalues of
$F_q$ lie on the circle of radius $1/(2k|C_d|^2)$ centered in
$\rmi/(2k|C_d|^2)$ in the complex plane (cf.,
e.g.,~\cite[pp.~285--286]{colton2013inverse}). 

By linearity, for any given function $g\in\Lt(\Sd)$, the solution to
the direct scattering problem~\eqref{eq:ScatteringProblem} with
incident field 
\begin{subequations}
  \label{eq:Herglotz}
  \begin{equation}
    \label{eq:HerglotzA}
    \ui_g(x) 
    \,=\, \int_\Sd e^{\rmi k x\cdot\theta} g(\theta) \ds(\theta) \,, 
    \qquad x\in\Rd \,,
  \end{equation}
  is given by 
  \begin{equation}
    \label{eq:HerglotzB}
    u_{q,g}(x) 
    \,=\, \int_\Sd u_q(x;\theta) g(\theta) \ds(\theta) \,, 
    \qquad x\in\Rd \,,
  \end{equation}
  and the corresponding scattered field
  \begin{equation}
    \label{eq:HerglotzC}
    \us_{q,g}(x) 
    \,=\, \int_\Sd \us_q(x;\theta) g(\theta) \ds(\theta) \,, 
    \qquad x\in\Rd \,,
  \end{equation}
  has the far field pattern $\uinfty_{q,g}=F_qg$ satisfying
  \begin{equation}
    \label{eq:HerglotzD}
    \uinfty_{q,g}(\xhat) 
    \,=\, \int_{\di B_R(0)} \Bigl( \us_{q,g}(y) 
    \frac{\di e^{-\rmi k \xhat\cdot y}}{\di\nu_y} 
    - e^{-\rmi k \xhat\cdot y} 
    \frac{\di \us_{q,g}}{\di\nu}(y) \Bigr) \ds(y) \,,
    \qquad \xhat\in\Sd \,.
  \end{equation}
\end{subequations}
Incident fields as in \eqref{eq:HerglotzA} are usually called 
\emph{Herglotz wave functions}.

\section{A monotonicity relation for the far field operator}
\label{sec:Monotonicity}
We will frequently be discussing relative orderings compact
self-adjoint operators. 
The following extension of the Loewner order was introduced in
\cite{harrach2017monotonicity}. 
Let  $A,B:X\to X$ be two compact self-adjoint linear operators on a
Hilbert space $X$. 
We write
\begin{equation*}
  A \,\leqd\, B \qquad \text{for some $r\in\N$} \,,
\end{equation*}
if $B-A$ has at most $r$ negative eigenvalues.
Similarly, we write $A\leqfin B$ if $A\leqd B$ holds for some
$r\in\N$, and the notations $A\geqd B$ and $A\geqfin B$ are defined
accordingly. 

The following result was shown in
\cite[Cor.~3.3]{harrach2017monotonicity}. 

\begin{lemma}
  \label{lmm:leqfin}
  Let $A,B:X\to X$ be two compact self-adjoint linear operators on a
  Hilbert space $X$ with scalar product $\langle\cdot,\cdot\rangle$,
  and let $r\in\N$.
  Then the following statements are equivalent:
  \begin{itemize}
  \item[(a)] $A\leqd B$
  \item[(b)] There exists a finite-dimensional subspace $V\tm X$ with
    $\dim(V)\leq r$ such that
    \begin{equation*}
      \langle (B-A)v,v\rangle \,\geq\, 0 
      \qquad\text{for all } v\in\Vperp \,.  
    \end{equation*}
  \end{itemize}
\end{lemma}

In particular this lemma shows that $\leqfin$ and $\geqfin$ are
transitive relations (see \cite[Lmm.~3.4]{harrach2017monotonicity})
and thus preorders. 
We use this notation in the following \emph{monotonicity relation} for
the far field operator. 

\begin{theorem}
  \label{thm:Monotonicity}
  Let $q_1,q_2\in\LinftyCRd$.
  Then there exists a finite-dimensional subspace $V\tm\LtSd$ such that
  \begin{equation}
    \label{eq:Monotonicity1a}
    \real\Bigl(\int_\Sd g\, \ol{\Scal_{q_1}^*(F_{q_2}-F_{q_1})g} \ds\Bigr)
    \,\geq\, k^2 \int_{\Rd} (q_2-q_1) |u_{q_1,g}|^2 \dx 
    \quad \text{for all $g\in \Vperp$} \,.
  \end{equation}
  In particular
  \begin{equation}
    \label{eq:Monotonicity1b}
    q_1 \leq q_2 \qquad \text{implies that} \qquad 
    \real(\Scal_{q_1}^*F_{q_1}) 
    \leqfin \real(\Scal_{q_1}^*F_{q_2}) \,,
  \end{equation}
  where as usual the real part of a linear operator
  $A:X\to X$ on a Hilbert space $X$ is the self-adjoint operator given
  by $\real(A):=\frac12(A+A^*)$.
\end{theorem}

\begin{remark}
  \label{rem:Monotonicity2}
  Since the scattering operators $\Scal_1$ and $\Scal_2$ are unitary,
  we find using~\eqref{eq:ScatteringOperator} that 
  \begin{equation*}
    \begin{split}
      &\Scal_{q_1}^*(F_{q_2}-F_{q_1})
      \,=\, \frac1{2\rmi k|C_d|^2} 
      \Scal_{q_1}^*(\Scal_{q_2}-\Scal_{q_1})
      \,=\, \frac1{2\rmi k|C_d|^2} (\Scal_{q_1}^*\Scal_{q_2}-I)\\
      &\,=\, \Bigl( \frac1{2\rmi k|C_d|^2} 
      (I-\Scal_{q_2}^*\Scal_{q_1}) \Bigr)^*
      \,=\, \Bigl( \frac1{2\rmi k|C_d|^2} 
      \Scal_{q_2}^*(\Scal_{q_2}-\Scal_{q_1})  \Bigr)^*
      \,=\, \bigl( \Scal_{q_2}^*(F_{q_2}-F_{q_1}) \bigr)^* \,.
    \end{split}
  \end{equation*}
  Recalling that the eigenvalues of a compact linear operator and of its
  adjoint are complex conjugates of each other, we conclude that the
  spectra of $\real(\Scal_{q_1}^*(F_{q_2}-F_{q_1}))$ and 
  $\real(\Scal_{q_2}^*(F_{q_2}-F_{q_1}))$ coincide. 
  Consequently, the monotonicity relations
  \eqref{eq:Monotonicity1a}--\eqref{eq:Monotonicity1b} remain true, if
  we replace $\Scal_{q_1}^*$ by $\Scal_{q_2}^*$ in these formulas.
  \hfill$\lozenge$
\end{remark}

Interchanging the roles of $q_1$ and $q_2$, except for $\Scal_{q_1}^*$
(see Remark~\ref{rem:Monotonicity2}), we may restate 
Theorem~\ref{thm:Monotonicity} as follows.

\begin{corollary}
  \label{cor:Monotonicity}
  Let $q_1,q_2\in\LinftyCRd$. 
  Then there exists a finite-dimensional subspace ${V\tm\LtSd}$ such
  that 
  \begin{equation}
    \label{eq:Monotonicity1c}
    \real\Bigl(\int_\Sd g\, \ol{\Scal_{q_1}^* (F_{q_2}-F_{q_1})g} \ds\Bigr)
    \,\leq\, k^2 \int_{\Rd} (q_2-q_1) |u_{q_2,g}|^2 \dx 
    \quad \text{for all $g\in \Vperp$} \,.
  \end{equation}
\end{corollary}

\begin{remark}
  \label{rem:Monotonicity1}
  A well known monotonicity principle for the phases of the
  eigenvalues of the far field operator, which has been discussed,
  e.g., in \cite[Lmm.~4.1]{kirsch2013inside}, can be rephrased as 
  $\real(F_q)\geqfin0$ if $q>0$ and $\real(F_q)\leqfin0$ if $q<0$
  a.e.\ on the support of the contrast function $\supp(q)$.
  This result can now also be obtained as a special case of
  \eqref{eq:Monotonicity1a} in Theorem~\ref{thm:Monotonicity} with
  $q_1=0$ and $q_2=q$ if $q>0$ (or $q_1=q$ and $q_2=0$ and
  $\Scal_{q_1}^*$ replaced by $\Scal_{q_2}^*$ (see
  Remark~\ref{rem:Monotonicity2}) if $q<0$). 

  The monotonicity relation \eqref{eq:Monotonicity1b}, which is a
  consequence of the stronger result \eqref{eq:Monotonicity1a}, has
  already been established in
  \cite[Lmm.~3]{lakshtanov2016difference}, using rather different
  techniques.~$\lozenge$
\end{remark}

The proof of Theorem~\ref{thm:Monotonicity} is a simple corollary of
the following lemmas. 
We begin by summarizing some useful identities for the solution of the
scattering problem \eqref{eq:ScatteringProblem}.

\begin{lemma}
  \label{lmm:useful_relations}
  Let $q\in\LinftyCRd$, $n^2=1+q$, and let $\BR$ be a ball containing
  $\supp(q)$.  Then 
  \begin{equation}
    \label{eq:useful_1}
    \int_\Sd g\, \ol{F_{q}g}\ds 
    \,=\, k^2 \int_\BR q \ui_{g} \ol{u_{q,g}} \dx
    \qquad \text{for all } g\in\LtSd \,,
  \end{equation}
  and, for any $v\in H^1(\BR)$,
  \begin{equation}
    \label{eq:useful_2}
    \int_\BR \bigl(\nabla\us_{q,g}\cdot \nabla v - k^2n^2 \us_{q,g} v
    \bigr) \dx - \int_{\di\BR} v \frac{\di\us_{q,g}}{\di\nu}\ds
    \,=\, k^2 \int_\BR q\ui_{g} v\dx \,.    
  \end{equation}

  Furthermore, if $q_1,q_2\in\LinftyCRd$ and $\BR$ is a ball
  containing $\supp(q_1)\cup \supp(q_2)$, then, for any 
  $j,l\in \{1,2\}$,
  \begin{equation}
    \label{eq:useful_3}
    \int_{\di\BR} \Bigl(
    \us_{q_j,g} \ol{\frac{\di\us_{q_l,g}}{\di\nu}}
    - \ol{\us_{q_l,g}} \frac{\di\us_{q_j,g}}{\di\nu} \Bigr) \ds
    \,=\, -2\rmi k|C_d|^2 \int_\Sd F_{q_j}g\, \ol{F_{q_l}g} \ds \,,
  \end{equation}
  where $C_d$ denotes the constant from \eqref{eq:DefCd}.
\end{lemma}

\begin{proof}
  Let $g\in\LtSd$, then the scattered  field $\us_{q,g}\in
  \Heinsloc$ from \eqref{eq:HerglotzC} solves
  \begin{equation*}
    \Delta \us_{q,g} + k^2 n^2 \us_{q,g} 
    \,=\, -\Delta \ui_{q,g} - k^2 n^2 \ui_{q,g}
    \,=\, - k^2 q \ui_{q,g} \qquad \text{in $\BR$}\,,
  \end{equation*}
  and accordingly Green's formula shows that, for any $v\in H^1(\BR)$,
  \begin{equation*}
    \begin{split}
      \int_{\BR}\!\! \nabla \us_{q,g}\cdot \nabla v \dx 
      &\,=\, \int_{\di\BR}\!\! v \frac{\di\us_{q,g}}{\di\nu} \ds
      + k^2\! \int_{\BR}\!\! n^2 \us_{q,g} v \dx 
      + k^2\! \int_{\BR}\!\! q \ui_{q,g} v \dx \,,
    \end{split}
  \end{equation*}
  which proves \eqref{eq:useful_2}.

  Likewise, we obtain from \eqref{eq:HerglotzD} and Green's formula
  that 
  \begin{equation*}
    \begin{split}
      \uinfty_{q,g}(\theta)
      &\,=\, \int_{\di\BR} \Bigl( \us_{q,g}(y) 
      \frac{\di e^{-\rmi k \theta\cdot y}}{\di\nu_y}
      - e^{-\rmi k \theta\cdot y} \frac{\di\us_{q,g}}{\di\nu}(y) \Bigr) 
      \ds(y)\\
      &\,=\, k^2 \int_\BR q(y) u_{q,g}(y) e^{-\rmi k \theta\cdot y} \dy \,,
    \end{split}
  \end{equation*}
  and thus
  \begin{equation*}
\begin{split}
    \int_\Sd g\, \ol{F_{q}g} \ds 
    &\,=\, k^2 \int_\BR  q(y) \ol{u_{q,g}(y)} 
    \int_\Sd g(\theta) e^{\rmi k \theta\cdot y} \ds(\theta) \dy \,. 
\end{split}
  \end{equation*}
  Using \eqref{eq:HerglotzA} this shows \eqref{eq:useful_1}.

  To see \eqref{eq:useful_3} let $r>R$, then
  $\us_{q_j,g},\us_{q_l,g}\in \Heinsloc$ solve 
  (for $q=q_j$ and $q=q_l$)
  \begin{equation*}
    \Delta \us_{q,g} + k^2 \us_{q,g} 
    \,=\, 0 
    \qquad \text{in $B_r(0)\setminus \ol{\BR}$}\,,
  \end{equation*}
  and applying Green's formula we obtain that
  \begin{equation}
    \label{eq:ProofUseful6}
      \int_{\di B_r(0)}\!\! \Bigl(
      \us_{q_j,g}\ol{\frac{\di\us_{q_l,g}}{\di\nu}} 
      -\ol{\us_{q_l,g}}\frac{\di\us_{q_j,g}}{\di\nu} 
      \Bigr) \ds
      \,=\,
      \int_{\di\BR}\!\! \Bigl(
      \us_{q_j,g}\ol{\frac{\di\us_{q_l,g}}{\di\nu}} 
      - \ol{\us_{q_l,g}}\frac{\di\us_{q_j,g}}{\di\nu} 
      \Bigr) \ds \,.
  \end{equation}
  Using the radiation condition \eqref{eq:Sommerfeld} and the far field
  expansion \eqref{eq:FarfieldExpansion} (for $q=q_j$ and $q=q_l$) we
  find that, as $r\to\infty$,   
  \begin{equation}
    \label{eq:ProofUseful8}
    \begin{split}
      \int_{\di B_r(0)} \Bigl(
      \us_{q_j,g}\ol{\frac{\di\us_{q_l,g}}{\di\nu}} 
      - \ol{\us_{q_l,g}}\frac{\di\us_{q_j,g}}{\di\nu} \Bigr) \ds
      &\,=\, -2 \rmi k \int_{\di B_r(0)} 
      \us_{q_j,g} \ol{\us_{q_l,g}} \ds + o(1)\\ 
      &\,=\, -2 \rmi k |C_d|^2 \int_{\Sd} 
      F_{q_j}g \,\ol{F_{q_l}g} \ds + o(1) \,.
    \end{split}
  \end{equation}
  Substituting \eqref{eq:ProofUseful8} into \eqref{eq:ProofUseful6} and
  letting $r\to\infty$ finally gives \eqref{eq:useful_3}.
\end{proof}

The next tool we will use to prove the monotonicity relation for the
far field operator in Theorem~\ref{thm:Monotonicity} is the following
integral identity.

\begin{lemma}
  \label{lmm:Monotonicity1}
  Let $q_1,q_2\in\LinftyCRd$ and let $\BR$ be a ball containing
  $\supp(q_1)\cup \supp(q_2)$. 
  Then, for any $g\in\LtSd$,
  \begin{multline}
    \int_\Sd\!\! (g \, \ol{F_{q_2} g} - \ol{g}\,  F_{q_1} g)\!\ds
    + 2\rmi k|C_d|^2\!\! \int_\Sd\!\!\! F_{q_1}g\, \ol{F_{q_2}g}\!\ds
    + k^2\!\!\! \int_\Rd\!\! (q_1-q_2)|u_{q_1,g}|^2 \!\dx\\
    \,=\, \int_\BR(|\nabla(\us_{q_2,g}-\us_{q_1,g})|^2
    -k^2n_2^2|\us_{q_2,g}-\us_{q_1,g}|^2)\dx\\
    -\int_{\di\BR}\ol{(\us_{q_2,g}-\us_{q_1,g})}
    \frac{\di(\us_{q_2,g}-\us_{q_1,g})}{\di\nu}\ds \,.
    \label{eq:Monotonicity_equality}  
  \end{multline}
\end{lemma}

\begin{proof}
  The identity \eqref{eq:useful_3} (with $j=1$ and $l=2$) immediately
  implies that\vspace*{-2em}
  \begin{multline*}
    2\real\int_{\di\BR} \ol{\us_{q_1,g}}\frac{\di\us_{q_2,g}}{\di\nu} \ds\\
    \,=\, \int_{\di\BR} \Bigl( \ol{\us_{q_1,g}}\frac{\di\us_{q_2,g}}{\di\nu} 
    + \ol{\us_{q_2,g}}\frac{\di\us_{q_1,g}}{\di\nu}\Bigr) \ds
    -2\rmi k|C_d|^2 \int_\Sd F_{q_1}g\, \ol{F_{q_2}g} \ds \,.
  \end{multline*}
  Using this and \eqref{eq:useful_2} we find that\vspace*{-1em}
  \begin{equation*}
    \begin{split}
      &\int_\BR \bigl(|\nabla \us_{q_2,g}-\nabla\us_{q_1,g}|^2 
      - k^2 n_2^2 |\us_{q_2,g}-\us_{q_1,g}|^2\bigr) \dx\\
      &\phantom{\,=\,}
      -\int_{\di\BR} \ol{(\us_{q_2,g}-\us_{q_1,g})}
      \frac{\di(\us_{q_2,g}-\us_{q_1,g})}{\di\nu} \ds\\
      &\,=\, \int_\BR \bigl(|\nabla \us_{q_2,g}|^2 
      -k^2n_2^2|\us_{q_2,g}|^2\bigr) \dx
      + \int_\BR \bigl(|\nabla \us_{q_1,g} |^2 
      - k^2 n_2^2 |\us_{q_1,g}|^2\bigr) \dx\\
      &\phantom{\,=\,}
      -2\real\Bigl( \int_\BR \bigl(\nabla\us_{q_2,g}\cdot\ol{\nabla\us_{q_1,g}} 
      -k^2 n_2^2\us_{q_2,g}\ol{\us_{q_1,g}}\bigr) \dx 
      -\int_{\di\BR} \ol{\us_{q_1,g}} \frac{\di\us_{q_2,g}}{\di\nu}\ds \Bigr)\\
      &\phantom{\,=\,} 
      + 2\rmi k|C_d|^2 \int_\Sd F_{q_1}g\, \ol{F_{q_2}g}\ds
      - \int_{\di\BR} \ol{\us_{q_2,g}} \frac{\di\us_{q_2,g}}{\di\nu} \ds
       - \int_{\di\BR} \ol{\us_{q_1,g}} \frac{\di\us_{q_1,g}}{\di\nu} \ds\\
      &\,=\, k^2 \int_\BR q_2 \ui_g\ol{\us_{q_2,g}} \dx
      - 2 \real\Bigl( k^2\int_\BR q_2 \ui_g \ol{\us_{q_1,g}}\dx\Bigr)
      + k^2 \int_\BR q_1 \ui_g\ol{\us_{q_1,g}}\dx\\
      &\phantom{\,=\,} 
      + k^2\int_\BR (q_1-q_2) |\us_{q_1,g}|^2\dx
      + 2\rmi k|C_d|^2 \int_\Sd F_{q_1}g\, \ol{F_{q_2}g}\ds \,.\\[8em]
    \end{split}
  \end{equation*}
  Further simple manipulations give
  \begin{equation*}
    \begin{split}
      &\int_\BR \bigl(|\nabla \us_{q_2,g}-\nabla\us_{q_1,g}|^2 
      - k^2 n_2^2 |\us_{q_2,g}-\us_{q_1,g}|^2\bigr) \dx\\ 
      &\phantom{\,=\,}   
      -\int_{\di\BR} \ol{(\us_{q_2,g}-\us_{q_1,g})}
      \frac{\di(\us_{q_2,g}-\us_{q_1,g})}{\di\nu} \ds\\
      &\,=\, k^2 \int_\BR q_2 \ui_g \ol{\us_{q_2,g}}\dx 
      - 2 \real\Bigl( k^2\int_\BR (q_2-q_1) \ui_g \ol{\us_{q_1,g}}\dx\Bigr)\\
      &\phantom{\,=\,} 
      - k^2\! \int_\BR\!\! q_1 \ol{\ui_g} \us_{q_1,g} \dx
      - k^2\! \int_\BR\!\! (q_2-q_1) |\us_{q_1,g}|^2 \dx
      + 2\rmi k|C_d|^2\! \int_\Sd\!\! F_{q_1}g\, \ol{F_{q_2}g}\ds\\
      &\,=\, k^2 \int_\BR q_2 \ui_g \ol{u_{q_2,g}} \dx
      - k^2 \int_\BR q_1 \ol{\ui_g} u_{q_1,g} \dx
      - k^2\int_\BR (q_2 - q_1) |\ui_g|^2 \dx\\
      &\phantom{\,=\,}  
      - 2 \real\Bigl( k^2\int_\BR (q_2-q_1) \ol{\ui_g} \us_{q_1,g}\dx\Bigr)
      - k^2 \int_\BR (q_2-q_1) |\us_{q_1,g}|^2\dx \\
      &\phantom{\,=\,} 
      + 2\rmi k|C_d|^2 \int_\Sd F_{q_1}g\, \ol{F_{q_2}g}\ds\\
      &\,=\, k^2 \int_\BR q_2 \ui_g \ol{u_{q_2,g}} \dx 
      - k^2 \int_\BR q_1 \ol{\ui_g} u_{q_1,g} \dx
      - k^2\int_\BR (q_2-q_1) |u_{q_1,g}|^2 \dx\\
      &\phantom{\,=\,}  
      + 2\rmi k|C_d|^2 \int_\Sd F_{q_1}g\, \ol{F_{q_2}g}\ds \,.
    \end{split}
  \end{equation*}\\
  Finally, applying \eqref{eq:useful_1} we obtain that
  \begin{equation*}
    \begin{split}
      &\int_\BR \bigl(|\nabla \us_{q_2,g}-\nabla\us_{q_1,g}|^2 
      - k^2 n_2^2 |\us_{q_2,g}-\us_{q_1,g}|^2\bigr) \dx\\
      &\phantom{\,=\,}  
      -\int_{\di\BR}
      \ol{(\us_{q_2,g}-\us_{q_1,g})}
      \frac{\di(\us_{q_2,g}-\us_{q_1,g})}{\di\nu} \ds\\
      &\,=\,\int_\Sd ( g \, \ol{F_{q_2} g} - \ol{g}\,  F_{q_1} g ) \ds 
      - k^2\int_\BR (q_2-q_1) |u_{q_1,g}|^2 \dx\\
      &\phantom{\,=\,}  
      + 2\rmi k|C_d|^2 \int_\Sd F_{q_1}g\, \ol{F_{q_2}g}\ds \,,
    \end{split}
  \end{equation*}
  which proves the assertion.
\end{proof}

\begin{remark}
  Since the adjoint of the scattering operator $\Scal_{q_1}$ from
  \eqref{eq:ScatteringOperator} is given by 
  \begin{equation*}
    \Scal_{q_1}^*
    \,=\, I-2\rmi k|C_d|^2F_{q_1}^* \,,
  \end{equation*}
  we find that
  \begin{equation*}
    \Scal_{q_1}^*(F_{q_2}-F_{q_1})
    \,=\, F_{q_2}-F_{q_1} -2\rmi k|C_d|^2
    (F_{q_1}^*F_{q_2}-F_{q_1}^*F_{q_1}) \,, 
  \end{equation*}
  and accordingly, 
  \begin{equation*}
    \real(\Scal_{q_1}^*(F_{q_2}-F_{q_1}))
    \,=\, \real(F_{q_2}-F_{q_1} - 2\rmi k|C_d|^2F_{q_1}^*F_{q_2}) \,. 
  \end{equation*}
  Therefore the real part of the first two terms on the left hand
  side of \eqref{eq:Monotonicity_equality} fulfills
  \begin{equation}
    \label{eq:RewriteRealF}
    \begin{split}
      \real\Bigl(
      &\int_\Sd (g \, \ol{F_{q_2} g} - \ol{g}\,  F_{q_1} g)\ds
      + 2\rmi k|C_d|^2 \int_\Sd F_{q_1}g\, \ol{F_{q_2}g}\ds \Bigr)\\
      &\,=\, \real\Bigl(
      \int_\Sd g \ol{(F_{q_2}-F_{q_1}-2\rmi k|C_d|^2F_{q_1}^*F_{q_2})g}
      \ds\Bigr)\\
      &\,=\, \real\Bigl( \int_\Sd g \, \ol{\Scal_{q_1}^*(F_{q_2}-F_{q_1}) g}
      \ds\Bigr) \,.
    \end{split}
  \end{equation}
  The operator $\Scal_{q_1}^*(F_{q_2}-F_{q_1})$ is compact and normal
  (see \cite[Lemma~1]{lakshtanov2016difference}). 
  \hfill$\lozenge$
\end{remark}

Next we consider the right hand side of
\eqref{eq:Monotonicity_equality}, and we show that it is nonnegative
if $g$ belongs to the complement of a certain finite dimensional 
subspace $V\tm\LtSd$.  
To that end we denote by $I:H^1(\BR)\to H^1(\BR)$
the identity operator, by $J:H^1(\BR)\to L^2(\BR)$ the compact
embedding, and accordingly we define, for any $q\in\LinftyCRd$ and any
ball $\BR$ containing $\supp(q)$, the operator ${K:H^1(\BR)\to H^1(\BR)}$ by
\begin{equation*}
  Kv \,:=\, J^*J v \,,
\end{equation*}
and $K_{q}:H^1(\BR)\to H^1(\BR)$ by
\begin{equation*}
  K_{q} v \,:=\, J^*((1+q)J v) \,.
\end{equation*}
Then $K$ and $K_{q}$ are compact self-adjoint linear operators, and,
for any $v\in H^1(\BR)$,
\begin{equation*}
  \bigl\langle (I-K-k^2K_{q})v,v\bigr\rangle_{H^1(\BR)}\\
  \,=\, \int_\BR(|\nabla v|^2
  -k^2 (1+q) |v|^2)\dx \,.
\end{equation*}

For $0<\eps<R$ we denote by $N_\eps:H^1(\BR) \to L^2(\di\BR)$
the bounded linear operator that maps $v\in H^1(\BR)$ to the normal
derivative $\di v_\eps/\di\nu$ on $\di\BR$ of the radiating solution
to the exterior boundary value problem 
\begin{equation*}
  \Delta v_\eps+k^2v_\eps 
  \,=\, 0 \quad\text{in }\Rd\setminus \ol{B_{R-\eps}(0)} \,, 
  \qquad v_\eps\,=\,v \quad\text{on }\di B_{R-\eps}(0) \,,
\end{equation*}
and 
${\Lambda: L^2(\di\BR)\to L^2(\di\BR)}$ denotes the compact
exterior Neumann-to-\linebreak[4]Dirichlet operator that maps 
$\psi\in L^2(\di\BR)$ to the trace $w|_{\di\BR}$ of the
radiating solution to 
\begin{equation*}
  \Delta w+k^2w \,=0\, \quad\text{in }\Rd\setminus\ol{\BR} \,, 
  \qquad \frac{\di w}{\di\nu}\,=\,\psi \quad\text{on }\di\BR \,,
\end{equation*}
(see, e.g., \cite[p.~51--55]{colton2013inverse}).
Then,
\begin{equation*}
  N_\eps v \,=\,  \frac{\di v}{\di\nu}\Big|_{\di\BR}
  \qquad\text{and}\qquad
  \Lambda N_\eps \,=\,  v|_{\di\BR} \,,
\end{equation*}
and accordingly
\begin{equation*}
  \bigl\langle N_\eps^*\Lambda N_\eps v, v\bigr\rangle_{H^1(\BR)}
  \,=\, \bigl\langle \Lambda N_\eps v, N_\eps v\bigr\rangle_{L^2(\di\BR)}
  \,=\, \int_{\di\BR}\ol{v} \frac{\di v}{\di\nu}\ds 
\end{equation*}
for any $v\in H^1(\BR)$ that can be extended to a radiating solution
of the Helmholtz equation 
\begin{equation*}
  \Delta v+k^2v
  \,=\, 0 \quad\text{in }\Rd\setminus\ol{B_{R-\eps}(0)} \,.
\end{equation*}
In particular this holds for $v=\us_{q_2,g}-\us_{q_1,g}$ if the ball
$B_{R-\eps}(0)$ contains $\supp(q_1)$ and $\supp(q_2)$. 

\begin{lemma}
  \label{lmm:Monotonicity3}
  Let $q_1,q_2\in\LinftyCRd$ and let $\BR$ be a ball containing
  $\supp(q_1)\cup \supp(q_2)$. Then there exists a finite dimensional
  subspace $V\subset \LtSd$ such that 
  \begin{multline*}
    \label{eq:Monotonicity3}
    \int_\BR(|\nabla(\us_{q_2,g}-\us_{q_1,g})|^2
    -k^2n_2^2|\us_{q_2,g}-\us_{q_1,g}|^2) \dx\\
    - \real\Bigl( \int_{\di\BR}\ol{(\us_{q_2,g}-\us_{q_1,g})}
    \frac{\di(\us_{q_2,g}-\us_{q_1,g})}{\di\nu}\ds \Bigr)
    \,\geq\, 0 \qquad \text{for all } g\in V^\perp \,.
  \end{multline*}
\end{lemma}

\begin{proof}
  Let $\eps>0$ be sufficiently small, so that $\supp(q_1)\cup
  \supp(q_2)\subset B_{R-\eps}(0)$. 
  Then
  \begin{multline*}
      \int_\BR(|\nabla(\us_{q_2,g}-\us_{q_1,g})|^2
      -k^2n_2^2|\us_{q_2,g}-\us_{q_1,g}|^2) \dx\\
      - \real\Bigl(\int_{\di\BR}\ol{(\us_{q_2,g}-\us_{q_1,g})}
      \frac{\di(\us_{q_2,g}-\us_{q_1,g})}{\di\nu}\ds\Bigr)\\
      \,=\,
      \bigl\langle(I-K-k^2K_{q_2}-\real(N_\eps^*\Lambda N_\eps))(S_2-S_1)g,
      (S_2-S_1)g\bigr\rangle_{H^1(\BR)} \,,
  \end{multline*}
  where for $j=1,2$ we denote by $S_j:\LtSd\to H^1(\BR)$ the bounded
  linear operator that maps $g\in \LtSd$ to the restriction of the
  scattered field $\us_{q_j,g}$ on~$\BR$.

  Let $W$ be the sum of eigenspaces of the compact self-adjoint
  operator $K+k^2K_{q_2}+\real(N_\eps^*\Lambda N_\eps)$ associated to
  eigenvalues larger than $1$. Then $W$ is finite dimensional and 
  \begin{equation*}
    \bigl\langle(I-K-k^2K_{q_2}-\real(N_\eps^*\Lambda N_\eps))w, 
    w\bigr\rangle_{H^1(\BR)}
    \,\geq\, 0 
    \qquad \text{ for all } w\in W^\perp.
  \end{equation*}
  Since, for any $g\in\LtSd$,
  \begin{equation*}
    (S_2-S_1)g\in W^\perp 
    \qquad\text{if and only if}\qquad
    g\in \bigl((S_2-S_1)^*W\bigr)^\perp \,,
  \end{equation*}
  and of course $\dim((S_2-S_1)^*W)\leq\dim(W)<\infty$, choosing
  $V:=(S_2-S_1)^*W$ 
  ends the proof. 
\end{proof}

\begin{proof}[Proof of Theorem~\ref{thm:Monotonicity}]
  Taking the real part of \eqref{eq:Monotonicity_equality} and
  applying \eqref{eq:RewriteRealF}, the result follows immediately
  from Lemma~\ref{lmm:Monotonicity3}.
\end{proof}

\section{Localized wave functions}
\label{sec:Localized}
In this section we establish the existence of 
\emph{localized wave functions} that have arbitrarily large norm
on some prescribed region~$B\tm\Rd$ while at the same time having
arbitrarily small norm in a different region $D\tm\Rd$, assuming that 
$\Rd\setminus\ol{D}$ is connected.  
These will be utilized to establish a rigorous characterization of the
support of scattering objects in terms of the far field operator using
the monotonicity relations from Theorem~\ref{thm:Monotonicity} and
Corollary~\ref{cor:Monotonicity} in
Section~\ref{sec:ShapeReconstruction} below. 

\begin{theorem}
  \label{thm:LocPot1}
  Suppose that $q\in\LinftyCRd$ and
  let $B,D\tm\Rd$ be open and bounded such that $\Rd\setminus\ol{D}$
  is connected. 

  If $B\not\tm D$, then for any finite dimensional subspace
  $V\tm\LtSd$ there exists a sequence $(g_m)_{m\in\N}\tm\Vperp$ such
  that 
  \begin{equation*}
    \int_B |u_{q,g_m}|^2\dx \to \infty 
    \qquad\text{and}\qquad
    \int_D |u_{q,g_m}|^2\dx \to 0 \qquad \text{as } m\to\infty \,,
  \end{equation*}
  where $u_{q,g_m}\in \Heinsloc$ is given by \eqref{eq:HerglotzB}
  with $g=g_m$.
\end{theorem}

The proof of Theorem~\ref{thm:LocPot1} relies on the following
lemmas. 

\begin{lemma}
  \label{lmm:LocPot1}
  Suppose that $q\in\LinftyCRd$, let
  $n^2=1+q$, and assume that $D\tm\Rd$ is open and bounded.  We define 
  \begin{equation*}
    L_{q,D}:\LtSd\to\Lt(D) \,, \quad g\mapsto u_{q,g}|_D \,,
  \end{equation*}
  where $u_{q,g}\in \Heinsloc$ is given by \eqref{eq:HerglotzB}.
  Then $L_{q,D}$ is a compact linear operator and its adjoint is given by
  \begin{equation*}
    L_{q,D}^*: \Lt(D)\to\LtSd \,, \quad f\mapsto \Scal_q^*w^\infty \,,
  \end{equation*}
  where $\Scal_q$ denotes the scattering operator from
  \eqref{eq:ScatteringOperator}, and ${w^\infty\in\LtSd}$ is the far
  field pattern of the radiating solution $w\in\Heinsloc$
  to\footnote{Throughout, we identify $f\in\Lt(D)$ with its
    continuation to $\Rd$ by zero whenever appropriate.}
  \begin{equation}
    \label{eq:Defw}
    \Delta w+k^2n^2w \,=\, -f \qquad\text{in }\Rd \,.
  \end{equation}
\end{lemma}

\begin{proof}
  The representation formula for the total field in
  \eqref{eq:HerglotzB} shows that $L_{q,D}$ is a Fredholm integral
  operator with square integrable kernel and therefore compact and
  linear from $\LtSd$ to $\Lt(D)$.

  The existence and uniqueness of a radiating solution
  ${w\in\Heinsloc}$ of \eqref{eq:Defw} follows again
  from~\cite[Thm.~8.7]{colton2013inverse} (see also 
  \cite[Thm.~6.9]{kirsch2011introduction}). 
  To determine the adjoint of $L_{q,D}$ we first observe that, for any
  ball $\BR$, this solution satisfies, for any $v\in\Heins(\BR)$, 
  \begin{equation}
    \label{eq:ProofLocPot1-1}
    \int_\BR (\nabla w\cdot\nabla v -k^2n^2wv)\dx  
    \,=\, \int_\BR f v \dx + \int_{\di\BR} v \frac{\di w}{\di\nu}\ds \,.
  \end{equation}
  We choose $R>0$ large enough such that $\supp(q)$ and $D$ are
  contained in $\BR$.  
  Applying \eqref{eq:ProofLocPot1-1}, Green's formula, and the
  representation formula for the far field pattern~$w^\infty$ of $w$
  analogous to \eqref{eq:FarfieldPattern} we find that, for any
  $g\in\LtSd$ and $f\in\Lt(D)$, 
  \begin{equation}
    \label{eq:ProofLocpot1-1}
    \begin{split}
      &\int_D (L_{q,D}g) \ol{f} \dx
      \,=\, \int_\BR(\nabla u_{q,g}\cdot\ol{\nabla w}
      -k^2n^2u_{q,g}\ol{w})\dx
      -\int_{\di\BR} u_{q,g}\ol{\frac{\di w}{\di\nu}}\ds\\
      &\,=\, \int_{\di\BR} \Bigl(\frac{\di u_{q,g}}{\di\nu}\ol{w} 
      -u_{q,g}\ol{\frac{\di w}{\di\nu}}\Bigr)\ds\\
      &\,=\, \int_\Sd g(\theta) 
      \int_{\di\BR} \Bigl(\frac{\di e^{\rmi k\theta\cdot y}}{\di\nu_y}\ol{w(y)} 
      -e^{\rmi k\theta\cdot y}
      \ol{\frac{\di w}{\di\nu}(y)}\Bigr)\ds(y)\ds(\theta)\\
      &\phantom{\,=\,}+\int_\Sd g(\theta) 
      \int_{\di\BR} \Bigl(\frac{\di \us_{q,g}}{\di\nu_y}(y;\theta)\ol{w(y)} 
      -\us_{q,g}(y;\theta)\ol{\frac{\di w}{\di\nu}(y)}\Bigr)
      \ds(y)\ds(\theta)\\
      &\,=\, \int_\Sd g(\theta) \ol{w^\infty(\theta)} \ds(\theta)\\
      &\phantom{\,=\,}+\int_\Sd g(\theta) 
      \int_{\di\BR} \Bigl(\frac{\di \us_{q,g}}{\di\nu_y}(y;\theta)\ol{w(y)} 
      -\us_{q,g}(y;\theta)\ol{\frac{\di w}{\di\nu}(y)}\Bigr)
      \ds(y)\ds(\theta) \,.
    \end{split}
  \end{equation}
  Using the radiation condition \eqref{eq:Sommerfeld} and the farfield
  expansion \eqref{eq:FarfieldExpansion} we obtain that, as $R\to\infty$,
  \begin{equation*}
    \begin{split}
      \int_{\di\BR} \Bigl(\frac{\di \us_{q,g}}{\di\nu_y}(y;\theta)\ol{w(y)} 
      -&\us_{q,g}(y;\theta)\ol{\frac{\di w}{\di\nu}(y)}\Bigr) \ds(y)\\
      &\,=\, 2\rmi k \int_{\di\BR} \us_{q,g}(y;\theta)\ol{w(y)} \ds(y) 
      + o(1)\\
      &\,=\, 2\rmi k |C_d|^2  
      \int_\Sd \uinfty_{q,g}(\yhat;\theta )\ol{w^\infty(\yhat)}\ds(\yhat) 
      + o(1) \,.
    \end{split}
  \end{equation*}
  Accordingly, substituting this into \eqref{eq:ProofLocpot1-1}, and
  using \eqref{eq:FarfieldOperator} and \eqref{eq:ScatteringOperator}
  gives 
  \begin{equation*}
    \begin{split}
      &\int_D (L_{q,D}g) \ol{f} \dx\\
      &\,=\, \int_\Sd g(\theta) \ol{w^\infty(\theta)} \ds(\theta)
      +2\rmi k |C_d|^2 \int_\Sd g(\theta) 
      \int_\Sd\uinfty_{q,g}(\yhat;\theta)\ol{w^\infty(\yhat)}\ds(\yhat)
      \ds(\theta)\\
      &\,=\, \int_\Sd\!\! g(\theta) \ol{w^\infty(\theta)} \ds(\theta)
      +2\rmi k |C_d|^2 \int_\Sd\!\! (F_qg)(\yhat) 
      \ol{w^\infty(\yhat)} \ds(\yhat)
      \,=\, \int_\Sd\!\! g \, \ol{\Scal_q^*w^\infty} \ds \,.
    \end{split}
  \end{equation*}
\,
\end{proof}

\begin{lemma}
  \label{lmm:LocPot2}
  Suppose that $q\in\LinftyCRd$.  Let
  $B,D\tm\Rd$ be open and bounded such that
  $\Rd\setminus(\ol{B}\cup\ol{D})$ is connected and 
  $\ol{B}\cap\ol{D}=\emptyset$.  Then,
  \begin{equation*}
    \Rcal(L_{q,B}^*)\cap\Rcal(L_{q,D}^*) \,=\, \{0\} \,,
  \end{equation*}
  and $\Rcal(L_{q,B}^*),\Rcal(L_{q,D}^*)\tm\LtSd$ are both dense.
\end{lemma}

\begin{proof}
  To start with, we show the injectivity of $L_{q,B}$, and we note
  that the injectivity of $L_{q,D}$ follows analogously.
  Let $R>0$ such that $\supp(q)\tm\BR$.
  Then the solution $u_{q,g}$ of \eqref{eq:ScatteringProblem} from
  \eqref{eq:HerglotzB} satisfies the Lippmann-Schwinger equation
  \begin{equation}
    \label{eq:LippmannSchwinger}
    u_{q,g}(x) 
    \,=\, u_g^i(x) + k^2 \int_\Rd q(y) \Phi(x-y) u_{q,g}(y) \dy \,, 
    \qquad x\in\ \BR \,,
  \end{equation}
  where $\Phi$ denotes the fundamental solution to the
  Helmholtz equation (cf., e.g., \cite[Thm.~8.3]{colton2013inverse}). 
  By unique continuation, $L_{q,B}g=u_{q,g}|_B=0$ implies that
  $u_{q,g}=0$ in $\Rd$ (cf., e.g.,
  \cite[Sec.~2.3]{harrach2017monotonicity}). 
  Substituting this into \eqref{eq:LippmannSchwinger}, we find that the
  Herglotz wave function $u_g^i=0$ in $\BR$, and thus by
  analyticity on all of $\Rd$. 
  This implies that $g=0$ (cf., e.g.,
  \cite[Thm.~3.19]{colton2013inverse}), i.e., $L_{q,B}$ is injective.

  The injectivity of $L_{q,B}$ and $L_{q,D}$ immediately yields that
  $\Rcal(L_{q,B}^*)$ and $\Rcal(L_{q,D}^*)$ are dense in $\LtSd$.
  Next suppose that $h\in\Rcal(L_{q,B}^*)\cap\Rcal(L_{q,D}^*)$.  Then
  Lemma~\ref{lmm:LocPot1} shows that there exist $f_B\in\Lt(B)$, 
  $f_D\in\Lt(D)$, and $w_B,w_D\in\Heinsloc$ such that the far field
  patterns $w^\infty_B$ and $w^\infty_D$ of the radiating solutions to
  \begin{equation*}
    \Delta w_B+k^2(1+q)w_B \,=\, -f_B \qquad\text{and}\qquad
    \Delta w_D+k^2(1+q)w_D \,=\, -f_D \qquad\text{in } \Rd \,,
  \end{equation*}
  satisfy
  \begin{equation*}
    w^\infty_B \,=\, w^\infty_D \,=\, \Scal_qh \,.
  \end{equation*}
  Rellich's lemma and unique continuation guarantee that $w_B=w_D$ in
  $\Rd\setminus(\ol{B}\cup\ol{D})$ (cf., e.g., \cite[Thm.~2.14]{colton2013inverse}).
  Hence we may define $w\in\Heinsloc$ by
  \begin{equation*}
    w \,:=\,
    \begin{cases}
      w_B=w_D &\text{in }\Rd\setminus(\ol{B}\cup\ol{D})\,,\\
      w_B &\text{in }D\,,\\
      w_D &\text{in }B\,,
    \end{cases}
  \end{equation*}
  and $w$ is the unique radiating solution to
  \begin{equation*}
    \Delta w + k^2(1+q)w \,=\, 0 \qquad\text{in }\Rd \,.
  \end{equation*}
  Thus $w=0$ in $\Rd$, and since the scattering operator is unitary,
  this shows that $h=\Scal_q^*w^\infty=0$. 
\end{proof}

In the next lemma we quote a special case of Lemma~2.5 in
\cite{harrach2013monotonicity}.
\begin{lemma}
  \label{lmm:Workhorse}
  Let $X,Y$ and $Z$ be Hilbert spaces, and let $A:X\to Y$ and $B:X\to
  Z$ be bounded linear operators.  Then, 
  \begin{equation*}
    \exists C>0:\; \|Ax\|\leq C\|Bx\| \quad \forall x\in X 
    \qquad\text{if and only if}\qquad
    \Rcal(A^*)\tm\Rcal(B^*) \,.
  \end{equation*}
\end{lemma}

Now we give the proof of Theorem~\ref{thm:LocPot1}. 

\begin{proof}[Proof of Theorem~\ref{thm:LocPot1}]
  Suppose that $q\in\LinftyCRd$, let
  $B,D\tm\Rd$ be open such that $\Rd\setminus\ol{D}$ is
  connected, and let $V\tm\LtSd$ be a finite dimensional subspace. 
  We first note that without loss of generality we may assume that
  $\ol{B}\cap\ol{D}=\emptyset$ and
  that~${\Rd\setminus(\ol{B}\cup\ol{D})}$ is connected (otherwise we
  replace $B$ by a sufficiently small 
  ball ${\widetilde{B}\tm B\setminus\ol{D_\eps}}$, where $D_\eps$
  denotes a sufficiently small neighborhood of~$D$).

  We denote by $P_V:\LtSd\to\LtSd$ the orthogonal projection on $V$.
  Lemma~\ref{lmm:LocPot2} shows that 
  $\Rcal(L_{q,B}^*)\cap\Rcal(L_{q,D}^*)=\{0\}$, and that $\Rcal(L_{q,B}^*)$ is
  infinite dimensional. Using a simple dimensionality argument
  (see~\cite[Lemma~4.7]{harrach2017monotonicity}) 
  it follows that
  \begin{equation*}
    \Rcal(L_{q,B}^*) 
    \,\not\tm\, \Rcal(L_{q,D}^*)+V 
    \,=\, \Rcal(
    \begin{pmatrix}
      L_{q,D}^* & P_V^*
    \end{pmatrix})
    \,=\, \Rcal\biggl(
    \begin{pmatrix}
      L_{q,D}\\ P_V
    \end{pmatrix}^*\biggr) \,.
  \end{equation*}
  Accordingly, Lemma~\ref{lmm:Workhorse} implies that there is no
  constant $C>0$ such that 
  \begin{equation*}
    \|L_{q,B}g\|^2_{\Lt(B)}
    \,\leq\, C^2\biggl\|
    \begin{pmatrix}
      L_{q,D}\\P_V
    \end{pmatrix}
    g\biggr\|^2_{\Lt(D)\times\Lt(\Sd)}
    \,=\, C^2\bigl( \|L_{q,D}g\|^2_{\Lt(D)}+\|P_Vg\|^2_{\Lt(\Sd)} \bigr)
  \end{equation*}
  for all $g\in\LtSd$.
  Hence, there exists as sequence $(\gtilde_m)_{m\in\N}\tm\LtSd$ such
  that 
  \begin{equation*}
    \|L_{q,B}\gtilde_m\|_{\Lt(B)}\to\infty \quad\text{and}\quad
    \|L_{q,D}\gtilde_m\|_{\Lt(D)} + \|P_V\gtilde_m\|_{\LtSd}\to 0 
    \qquad\text{as } m\to\infty\,.
  \end{equation*}
  Setting $g_m:=\gtilde_m-P_V\gtilde_m\in \Vperp\tm\LtSd$ for any
  $m\in\N$, we finally obtain
  \begin{align*}
    \|L_{q,B}g_m\|_{\Lt(B)} 
    \,\geq\,
    \|L_{q,B}\gtilde_m\|_{\Lt(B)}-\|L_{q,B}\|\|P_V\gtilde_m\|_{\LtSd}
    &\,\to\,\infty 
    &&\text{as } m\to\infty \,,\\
    \|L_{q,D}g_m\|_{\Lt(D)} 
    \,\leq\, \|L_{q,D}\gtilde_m\|_{\Lt(D)}+\|L_{q,D}\|\|P_V\gtilde_m\|_{\LtSd} 
    &\,\to\, 0 
    &&\text{as } m\to\infty \,.
  \end{align*}
  Since $L_{q,B}g_m=u_{q,g_m}|_B$ and $L_{q,D}g_m=u_{q,g_m}|_D$, this
  ends the proof. 
\end{proof}

The next result is a simple consequence of the
Lemmas~\ref{lmm:LocPot1} and \ref{lmm:Workhorse}. 
\begin{theorem}
  \label{thm:LocPot2}
  Suppose that $q_1,q_2\in\LinftyCRd$, and let $D\tm\Rd$ be open and
  bounded. 
  If $q_1(x)=q_2(x)$ for a.e.\ $x\in\Rd\setminus\ol{D}$, then
  there exist constants $c,C>0$ such that 
  \begin{equation*}
    \label{eq:LocPot2}
    c\int_D|u_{q_1,g}|^2\dx 
    \,\leq\, \int_D|u_{q_2,g}|^2\dx
    \,\leq\, C\int_D|u_{q_1,g}|^2\dx 
    \qquad\text{for all } g\in\LtSd \,,
  \end{equation*}
  where $u_{q_j,g}\in\Heinsloc$, $j=1,2$, is given by
  \eqref{eq:HerglotzB} with $q=q_j$.
\end{theorem}

\begin{proof}
  Let $q_1,q_2\in\LinftyCRd$.  We denote by
  $L_{q_1,D}$ and $L_{q_2,D}$ the operators from
  Lemma~\ref{lmm:LocPot1} with $q=q_1$ and $q=q_2$, respectively.  
  We showed in Lemma~\ref{lmm:LocPot1} that for any $f\in\Lt(D)$ 
  \begin{equation}
    \label{eq:ProofCorLocPot1}
    L_{q_1,D}^*f \,=\, \Scal_{q_1}^*w_1^\infty \qquad\text{and}\qquad
    L_{q_2,D}^*f \,=\, \Scal_{q_2}^*w_2^\infty \,,
  \end{equation}
  where $w_j^\infty$, $j=1,2$, are the far field patterns of the
  radiating solutions to 
  \begin{equation*}
    \Delta w_j + k^2(1+q_j)w_j \,=\, -f \qquad\text{in }\Rd \,.
  \end{equation*}
  This implies that
  \begin{subequations}
    \label{eq:ProofCorLocPot3}
    \begin{align}
      \Delta w_1 + k^2(1+q_2)w_1 &\,=\, -(f + k^2(q_1-q_2)w_1) 
                                   \qquad\text{in }\Rd \,,\\
      \Delta w_2 + k^2(1+q_1)w_2 &\,=\, -(f + k^2(q_2-q_1)w_2) 
                                   \qquad\text{in }\Rd \,.
    \end{align}
  \end{subequations}
  Since $q_1-q_2$ vanishes a.e.\ outside $D$, we find
  that 
  \begin{equation*}
    \Scal_{q_2}^* w_1^\infty \,=\, L_{q_2,D}^*(f+k^2(q_1-q_2)w_1) 
    \qquad\text{and}\qquad
    \Scal_{q_1}^* w_2^\infty \,=\, L_{q_1,D}^*(f+k^2(q_2-q_1)w_2) \,.
  \end{equation*}
  Combining \eqref{eq:ProofCorLocPot1} and \eqref{eq:ProofCorLocPot3},
  we obtain that
  $\Rcal(\Scal_{q_1}L_{q_1,D}^*)=\Rcal(\Scal_{q_2}L_{q_2,D}^*)$. 
  Since $\Scal_{q_1}$ and $\Scal_{q_2}$ are unitary operators, the
  assertion follows from Lemma~\ref{lmm:Workhorse}. 
\end{proof}

As a first application of Theorem~\ref{thm:LocPot1} we establish a
converse of \eqref{eq:Monotonicity1b} in
Theorem~\ref{thm:Monotonicity}. 

\begin{theorem}
  \label{thm:Converse}
  Suppose that $q_1,q_2\in\LinftyCRd$ with $\supp(q_j)\tm\BR$.
  If $O\tm\Rd$ is an unbounded domain such that
  \begin{equation*}
    q_1 \leq q_2 \qquad \text{a.e.\ in } O \,,
  \end{equation*}
  and if  $B\tm\BR\cap O$ is open with
  \begin{equation}
    \label{eq:Converse2}
    q_1 \leq q_2-c \qquad \text{a.e.\ in } B \text{ for some } c>0 \,,
  \end{equation}
  then
  \begin{equation*}
    \real(\Scal_{q_1}^*F_{q_1}) 
    \not\geqfin \real(\Scal_{q_1}^*F_{q_2}) \,,
  \end{equation*}
  i.e., the operator $\real(\Scal_{q_1}^*(F_{q_2}-F_{q_1}))$ has
  infinitely many positive eigenvalues. 
  In particular, this implies that $F_{q_1}\not=F_{q_2}$.
\end{theorem}

\begin{proof}
  We prove the result by contradiction and assume that 
  \begin{equation}
    \label{eq:ProofConverse1}
    \real(\Scal_{q_1}^*(F_{q_2}-F_{q_1})) \,\leqfin\, 0 \,.
  \end{equation}
  Using the monotonicity relation \eqref{eq:Monotonicity1a} in
  Theorem~\ref{thm:Monotonicity}, we find 
  that there exists a finite dimensional subspace $V\tm\LtSd$ such
  that
  \begin{equation}
    \label{eq:ProofConverse2}
    \real\left(\int_\Sd g\, 
      \ol{\Scal_{q_1}^*(F_{q_2}-F_{q_1})g} \ds\right)
    \,\geq\, k^2 \int_\BR (q_2-q_1) |u_{q_1,g}|^2 \dx
    \qquad\text{for all } g\in\Vperp \,.
  \end{equation}
  Combining \eqref{eq:ProofConverse1}, \eqref{eq:ProofConverse2}, and
  \eqref{eq:Converse2} we obtain that there exists a finite
  dimensional subspace $\Vtilde\tm\LtSd$ such that, for any
  $g\in\Vtilde^\perp$, 
  \begin{equation*}
    \begin{split}
      0
      &\,\geq\, \real\left(\int_\Sd g\,
        \ol{\Scal_{q_1}^*(F_{q_2}-F_{q_1})g} \ds\right)
      \,\geq\, k^2 \int_\BR (q_2-q_1) |u_{q_1,g}|^2 \dx\\
      &\,=\, k^2 \int_{O\cap\BR} (q_2-q_1) |u_{q_1,g}|^2 \dx
      + k^2 \int_{\BR\setminus \ol{O}} (q_2-q_1) |u_{q_1,g}|^2 \dx\\
      &\,\geq\, c k^2 \int_{B} |u_{q_1,g}|^2 \dx
      - C k^2 \int_{\BR\setminus \ol{O}} |u_{q_1,g}|^2 \dx \,,
    \end{split}
  \end{equation*}
  where $C:=\|q_1\|_{\Linfty(\Rd)}+\|q_2\|_{\Linfty(\Rd)}$. 
  However, this contradicts Theorem~\ref{thm:LocPot1} with
  ${D=\BR\setminus\ol{O}}$ and $q=q_1$, which guarantees the existence
  of ${(g_m)_{m\in\N}\tm\Vtilde^\perp}$ with 
  \begin{equation*}
    \int_B|u_{q_1,g_m}|^2\dx \to \infty
    \qquad\text{and}\qquad
    \int_{\BR\setminus \ol{O}}|u_{q_1,g_m}|^2\dx \to 0 
    \qquad\text{as } m\to\infty \,.
  \end{equation*}
  Consequently, $\real(\Scal_{q_1}^*(F_{q_2}-F_{q_1})) \not\leqfin 0$.
\end{proof}

\section{Monotonicity based shape reconstruction}
\label{sec:ShapeReconstruction}
Given any open and bounded subset $B\tm\Rd$, we define the operator 
$T_B:\LtSd \to \LtSd$ by
\begin{equation}
  \label{eq:DefTB1}
  T_Bg \,:=\, k^2 H_B^*H_Bg \,,
\end{equation}
where $H_B:L^2(\Sd)\to L^2(B)$ denotes the \emph{Herglotz operator}
given by
\begin{equation*}
  (H_Bg)(x) 
  \,:=\, \int_\Sd e^{\rmi k x\cdot\theta} g(\theta) \ds(\theta) \,,
  \qquad x\in B \,.
\end{equation*}
Accordingly, 
\begin{equation*}
  \int_\Sd g\ol{T_Bg}\ds 
  \,=\, k^2\int_B |\ui_g|^2 \dx \qquad 
  \text{for all } g\in\LtSd \,,
\end{equation*}
where $u^i_g$ denotes the Herglotz wave function with density
$g$ from \eqref{eq:HerglotzA}. 
The operator~$T_B$ is bounded, compact and self-adjoint, and it
coincides with the Born approximation of the far field operator $F_q$
with contrast function $q=\chi_B$, where $\chi_B$ denotes the
characteristic function of $B$ (see, e.g., \cite{kirsch2017remarks}). 

In the following we discuss criteria to determine the
support $\supp(q)$ of an unknown scattering object in terms of the
corresponding far field operator $F_q$.  
To begin with we discuss the case when the contrast function $q$ is
positive a.e.\ on its support.

\begin{theorem}
  \label{thm:Shape1}
  Let $B,D\tm\Rd$ be open and bounded such that
  $\Rd\setminus\ol{D}$ is connected, and let ${q\in\LinftyCRd}$ with
  $\supp(q)=\ol{D}$.
  Suppose that $0\leq\qmin\leq q\leq\qmax<\infty$ a.e.\ in $D$ for some
  constants $\qmin,\qmax\in\R$.
  \begin{itemize}
  \item[(a)] If $B\tm D$, then
    \begin{equation*}
      \alpha T_B \,\leqfin\, \real(F_q) 
      \qquad\text{for all } \alpha\leq\qmin \,.
    \end{equation*}
  \item[(b)] If $B\not\tm D$, then
    \begin{equation*}
      \alpha T_B \,\not\leqfin\, \real(F_q) 
      \qquad\text{for any } \alpha>0 \,,
    \end{equation*}
    i.e., the operator $\real(F_q)-\alpha T_B$ has infinitely many
    negative eigenvalues for all $\alpha>0$.
  \end{itemize}
\end{theorem}

\begin{proof}
  From Theorem~\ref{thm:Monotonicity} with $q_1=0$ and $q_2=q$ we
  obtain that there exists a finite dimensional subspace $V\tm\LtSd$
  such that
  \begin{equation*}
    \real\biggl(\int_\Sd g\, \ol{F_qg}\ds\biggr)
    \,\geq\, k^2 \int_D q|\ui_g|^2\dx
    \qquad \text{for all } g\in\Vperp \,.
  \end{equation*}
  Moreover, if $B\tm D$ and $\alpha\leq\qmin$, then 
  \begin{equation*}
    \alpha\int_\Sd g\ol{T_Bg}\ds
    \,=\, k^2\int_B \alpha|\ui_g|^2\dx
    \,\leq\, k^2\int_D q|\ui_g|^2 \dx \,,
  \end{equation*}
  which shows part (a).

  We prove part (b) by contradiction.  Let $B\not\tm D$, $\alpha>0$,
  and assume that
  \begin{equation}
    \label{eq:ProofShape1-3}
    \alpha T_B \,\leqfin\, \real(F_q) \,.
  \end{equation}
  Using the monotonicity relation \eqref{eq:Monotonicity1c} in
  Corollary~\ref{cor:Monotonicity} with $q_1=0$ and $q_2=q$, we find
  that there exists a finite dimensional subspace $V\tm\LtSd$ such
  that 
  \begin{equation}
    \label{eq:ProofShape1-4}
    \real\biggl(\int_\Sd g \, \ol{F_qg} \ds\biggr)
    \,\leq\, 
    k^2 \int_D q |u_{q,g}|^2\dx 
    \qquad\text{for all } g\in\Vperp \,.  
  \end{equation}
  Combining \eqref{eq:ProofShape1-3} and \eqref{eq:ProofShape1-4}, we
  obtain that there exists a finite dimensional
  subspace~$\Vtilde\in\LtSd$ such that 
  \begin{equation*}
    k^2 \alpha \int_B |\ui_g|^2 \dx
    \,\leq\, k^2 \int_D q |u_{q,g}|^2\dx 
    \,\leq\, k^2 \qmax \int_D |u_{q,g}|^2\dx 
    \qquad\text{for all } g\in\Vtilde^\perp \,. 
  \end{equation*}
  Applying Theorem~\ref{thm:LocPot2} with $q_1=0$ and $q_2=q$, this
  implies that there exists a constant~$C>0$ such that
  \begin{equation*}
    k^2 \alpha \int_B |\ui_g|^2 \dx
    \,\leq\, C k^2 \qmax \int_D |\ui_g|^2 \dx
    \qquad\text{for all } g\in\Vtilde^\perp \,.
  \end{equation*}
  However, this contradicts Theorem~\ref{thm:LocPot1} with $q=0$, which
  guarantees the existence of a sequence
  $(g_m)_{m\in\N}\tm\Vtilde^\perp$ with 
  \begin{equation*}
    \int_B|\ui_{g_m}|^2\dx \to \infty
    \qquad\text{and}\qquad
    \int_D|\ui_{g_m}|^2\dx \to 0 
    \qquad\text{as } m\to\infty \,.
  \end{equation*}
  Hence, $\real(F_q)-\alpha T_B$ must have infinitely many negative
  eigenvalues.  
\end{proof}

The next result is analogous to Theorem~\ref{thm:Shape2}, but with
contrast functions being negative on the support of the
scattering objects, instead of being positive.

\begin{theorem}
  \label{thm:Shape2}
  Let $B,D\tm\Rd$ be open and bounded such that
  $\Rd\setminus\ol{D}$ is connected, and let ${q\in\LinftyCRd}$ with
  $\supp(q)=\ol{D}$.
  Suppose that $-1<\qmin\leq q\leq\qmax\leq0$ a.e.\ in $D$ for some
  constants $\qmin,\qmax\in\R$.
  \begin{itemize}
  \item[(a)] If $B\tm D$, then there exists a constant $C>0$ such that
    \begin{equation*}
      \alpha T_B \,\geqfin\, \real(F_q) 
      \qquad\text{for all } \alpha\geq C\qmax \,.
    \end{equation*}
  \item[(b)]  If $B\not\tm D$, then
    \begin{equation*}
      \alpha T_B \,\not\geqfin\, \real(F_q) 
      \qquad\text{for any } \alpha<0 \,,
    \end{equation*}
    i.e., the operator $\real(F_q)-\alpha T_B$ has infinitely many
    positive eigenvalues.
  \end{itemize}
\end{theorem}

\begin{proof}
  If $B\tm D$, then Corollary~\ref{cor:Monotonicity} and
  Theorem~\ref{thm:LocPot2} with $q_1=0$ and $q_2=q$ show that there
  exists a constant $C>0$ and a finite dimensional subspace
  $V\tm\LtSd$ such that, for any $g\in\Vperp$,
  \begin{equation*}
    \begin{split}
      \real\biggl(\int_\Sd g \, \ol{F_qg}\ds\biggr)
      &\,\leq\, k^2 \int_D q |u_{q,g}|^2 \dx\\
      &\,\leq\, k^2 \qmax \int_D |u_{q,g}|^2 \dx
      \,\leq\, C k^2 \qmax \int_D |\ui_{g}|^2 \dx \,.
    \end{split}
  \end{equation*}
  In particular, 
  \begin{equation*}
    \real(F_q) \,\leqfin\, \alpha T_B 
    \qquad \text{for all } \alpha\geq C\qmax \,,
  \end{equation*}
  and part (a) is proven.

  We prove part (b) by contradiction.  
  Let $B\not\tm D$, $\alpha<0$, and assume that
  \begin{equation}
    \label{eq:ProofShape2-3}
    \alpha T_B \,\geqfin\, \real(F_q) \,.
  \end{equation}
  Using the monotonicity relation \eqref{eq:Monotonicity1a} in
  Theorem~\ref{thm:Monotonicity} with $q_1=0$ and $q_2=q$, we find
  that there exists a finite dimensional subspace $V\tm\LtSd$ such
  that 
  \begin{equation}
    \label{eq:ProofShape2-4}
    \real\biggl(\int_\Sd g\, \ol{F_qg}\ds\biggr)
    \,\geq\, k^2\int_D q |\ui_g|^2 \dx 
    \qquad \text{for all } g\in\Vperp \,. 
  \end{equation}
  Combining \eqref{eq:ProofShape2-3} and \eqref{eq:ProofShape2-4}
  shows that there exists a finite dimensional subspace
  ${\Vtilde\in\LtSd}$ such that
  \begin{equation*}
    k^2 \alpha \int_B |\ui_g|^2 \dx
    \,\geq\, k^2\int_D q|\ui_g|^2 \dx
    \,\geq\, k^2 \qmin \int_D |\ui_g|^2 \dx 
    \qquad \text{for all } g\in\Vtilde^\perp \,.
  \end{equation*}
  However, since $\alpha<0$, this contradicts
  Theorem~\ref{thm:LocPot1} with $q=0$, which guarantees the existence
  of a sequence $(g_m)_{m\in\N}\tm\Vtilde^\perp$ such that 
  \begin{equation*}
    \alpha \int_B|\ui_{g_m}|^2\dx \to -\infty 
    \qquad\text{and}\qquad
    \int_D|\ui_{g_m}|^2\dx \to 0 \,.
  \end{equation*}
  Hence, $\real(F_q)-\alpha T_B$ must have infinitely many positive
  eigenvalues for all $\alpha<0$.
\end{proof}

Next we consider the general case, i.e., the contrast function $q$ is
no longer required to be either positive or negative
a.e.\ on the support of all scattering objects.  
While in the \emph{sign definite case} the criteria developed in
Theorems~\ref{thm:Shape1}--\ref{thm:Shape2} determine whether a
certain \emph{probing domain} $B$ is contained in the support $D$ of
the scattering objects or not, the criterion for the \emph{indefinite
  case} established in Theorem~\ref{thm:Shape3} below characterizes
whether a certain probing domain $B$ contains the support $D$ of the
scattering objects or not. 

\begin{theorem}
  \label{thm:Shape3}
  Let $B,D\tm\Rd$ be open and bounded such that $\R^d\setminus\ol{D}$
  is connected, and let ${q\in\LinftyCRd}$ with $\supp(q)=\ol{D}$. 
  Suppose that $-1<\qmin\leq q\leq\qmax<\infty$ a.e.\ on $D$ for
  some constants $\qmin,\qmax\in\R$.

  Furthermore, we assume that for any point $x\in\di D$ on the
  boundary of $D$, and for any neighborhood $U\tm D$ of $x$ in $D$, there
  exists an unbounded neighborhood $O\tm\Rd$ of $x$ with 
  $O\cap D\tm U$, and an open subset $E\tm O\cap D$, such
  that\,\footnote{As usual, the inequalities in
    \eqref{eq:LocalDefiniteness} are to be understood pointwise almost
    everywhere.} 
  \begin{equation}
    \label{eq:LocalDefiniteness}
    q|_O\geq0 \;\text{ and }\; q|_E\geq\qminE>0
    \qquad\text{or}\qquad 
    q|_O\leq0 \;\text{ and }\; q|_E\leq\qmaxE<0 
  \end{equation}
  for some constants $\qminE,\qmaxE\in\R$.
  \begin{itemize}
  \item[(a)] If $D\tm B$, then there exists a constant $C>0$ such that
    \begin{equation*}
      \alpha T_B \,\leqfin\, \real(F_q) \,\leqfin\, \beta T_B
      \quad\text{for all } \alpha \leq \min\{0,\qmin\}\,,\; 
      \beta \geq \max\{0,C\qmax\} \,.
    \end{equation*}
  \item[(b)]  If $D\not\tm B$, then
    \begin{equation*}
      \alpha T_B \,\not\leqfin\, \real(F_q) 
      \quad\text{for any } \alpha\in\R
      \quad\text{or}\quad
      \real(F_q) \,\not\leqfin\, \beta T_B
      \quad\text{for any } \beta\in\R \,.
    \end{equation*}
  \end{itemize}
\end{theorem}

\begin{remark}
  The \emph{local definiteness property} \eqref{eq:LocalDefiniteness}
  in Theorem~\ref{thm:Shape3} is, e.g., always satisfied, if the
  contrast function is piecewise analytic (see Appendix~A of
  \cite{harrach2013monotonicity}) or if the supports of the positive part and of the
  negative part of the constrast function are well-separated from each
  other. 
\hfill$\lozenge$
\end{remark}

\begin{proof}[Proof of Theorem~\ref{thm:Shape3}]
  If $D\tm B$, then Corollary~\ref{cor:Monotonicity} and
  Theorem~\ref{thm:LocPot2} with $q_1=0$ and $q_2=q$ show that there
  exists a constant $C>0$ and a finite dimensional subspace
  $V\tm\LtSd$ such that, for all $g\in\Vperp$ and any
  $\beta\geq\max\{0,C\qmax\}$, 
  \begin{equation*}
    \begin{split}
      \real\biggl(\int_\Sd g \, \ol{F_qg}\ds\biggr)
      &\,\leq\, k^2 \int_D q |u_{q,g}|^2 \dx
      \,\leq\, k^2 \qmax \int_D |u_{q,g}|^2 \dx\\
      &\,\leq\, k^2 C \qmax \int_D |\ui_{g}|^2 \dx
      \,\leq\, k^2 \beta \int_B |\ui_{g}|^2 \dx \,.
    \end{split}
  \end{equation*}

  Similarly, Theorem~\ref{thm:Monotonicity} with $q_1=0$ and $q_2=q$
  shows that there exists a finite dimensional subspace $V\tm\LtSd$ 
  such that, for all $g\in\Vperp$ and any $\alpha\leq\min\{0,\qmin\}$, 
  \begin{equation*}
    \begin{split}
      \real\biggl(\int_\Sd g\, \ol{F_qg}\ds\biggr)
      &\,\geq\, k^2 \int_D q|\ui_g|^2\dx 
      \,\geq\, k^2 \qmin \int_D |\ui_g|^2\dx 
      \,\geq\, k^2 \alpha \int_B |\ui_g|^2\dx \,,
    \end{split}
  \end{equation*}
  and part (a) is proven.
  
  We prove part (b) by contradiction.  Since $D\not\tm B$,
  $U:=D\setminus B$ is not empty, and there exists 
  $x\in \ol{U}\cap\di D$ as well as an unbounded open neighborhood
  $O\tm\Rd$ of $x$ with $O\cap D\tm U$, and an open subset $E\tm O\cap
  D$ such that \eqref{eq:LocalDefiniteness} is satisfied.
  Furthermore, let $R>0$ be large enough such that $B,D\tm\BR$.
  
  We first assume that $q|_O\geq0$ and $q|_{B}\geq\qminE>0$, and that
  $\real(F_q)\leqfin\beta T_B$ for some $\beta\in\R$.
  Using the monotonicity relation \eqref{eq:Monotonicity1a} in
  Theorem~\ref{thm:Monotonicity} with $q_1=0$ and $q_2=q$, we find
  that there exists a finite dimensional subspace $V\tm\LtSd$ such
  that, for any $g\in\Vperp$,
  \begin{equation*}
    \begin{split}
      0
      &\,\geq\, \int_\Sd g(\ol{\real(F_q)g - \beta T_B g})\ds
      \,\geq\, k^2 \int_\BR (q-\beta\chi_B) |\ui_g|^2 \dx\\
      &\,=\, k^2 \int_{\BR\setminus\ol{O}} (q-\beta\chi_B) |\ui_g|^2 \dx
      + k^2 \int_{\BR\cap O} (q-\beta\chi_B) |\ui_g|^2 \dx\\
      &\,\geq\, -k^2(\|q\|_{\Linfty(\Rd)}+|\beta|)
      \int_{\BR\setminus\ol{O}} |\ui_g|^2 \dx
      + k^2 \qminE \int_{E} |\ui_g|^2 \dx \,. 
    \end{split}
  \end{equation*}
  However, this contradicts Theorem~\ref{thm:LocPot1} with $B=E$, 
  $D=\BR\setminus\ol{O}$, and $q=0$, which guarantees the existence of
  a sequence $(g_m)_{m\in\N}\tm\Vperp$ with 
  \begin{equation*}
    \int_E|\ui_{g_m}|^2\dx \to \infty
    \qquad\text{and}\qquad
    \int_{\BR\setminus\ol{O}}|\ui_{g_m}|^2\dx \to 0 
    \qquad\text{as } m\to\infty \,.
  \end{equation*}
  Consequently, $\real(F_q)\not\leqfin\beta T_B$ for all $\beta\in\R$.

  On the other hand, if $q|_O\leq0$ and $q|_{E}\leq\qmaxE<0$, and if
  $\alpha T_B\leqfin\real(F_q)$ for some $\alpha\in\R$, then the
  monotonicity relation \eqref{eq:Monotonicity1c} in
  Corollary~\ref{cor:Monotonicity} with 
  $q_1=0$ and $q_2=q$ shows that there exists a finite dimensional
  subspace $V\tm\LtSd$ such that, for any $g\in\Vperp$,
  \begin{equation*}
    \begin{split}
      0
      &\,\leq\, \int_\Sd g(\ol{\real(F_q)g - \alpha T_B g})\ds
      \,\leq\, k^2 \int_\BR (q |u_{q,g}|^2 
      - \alpha \chi_B|\ui_g|^2) \dx\\
      &\,=\, k^2 \int_{\BR\setminus\ol{O}} 
      (q |u_{q,g}|^2 - \alpha \chi_B|\ui_g|^2) \dx
      + k^2 \int_{\BR\cap O} 
      (q |u_{q,g}|^2 - \alpha \chi_B|\ui_g|^2) \dx\\
      &\,\leq\, k^2 \qmax \int_{\BR\setminus\ol{O}} |u_{q,g}|^2 \dx
      + k^2  |\alpha| \int_{\BR\setminus\ol{O}} |\ui_g|^2 \dx
      + k^2 \qmaxE \int_{E} |u_{q,g}|^2 \dx \,.
    \end{split}
  \end{equation*}
  Applying Theorem~\ref{thm:LocPot2} with $D=\BR\setminus\ol{O}$,
  $q_1=0$, and $q_2=q$ we find that there exists a constant $C>0$ such
  that
  \begin{equation*}
    \begin{split}
      0
      &\,\leq\, k^2 (C\qmax + |\alpha|) 
      \int_{\BR\setminus\ol{O}} |\ui_g|^2 \dx 
      + k^2 C \qmaxE \int_{E} |\ui_g|^2 \dx \,.
    \end{split}
  \end{equation*}
  However, since $\qmaxE<0$, this contradicts Theorem~\ref{thm:LocPot1}
  with $B=E$, $D=\BR\setminus\ol{O}$, and $q=0$, which guarantees the
  existence of a sequence $(g_m)_{m\in\N}\tm V^\perp$ with 
  \begin{equation*}
    \int_E|\ui_{g_m}|^2\dx \to \infty
    \qquad\text{and}\qquad
    \int_{\BR\setminus\ol{O}}|\ui_{g_m}|^2\dx \to 0 
    \qquad\text{as } m\to\infty \,.
  \end{equation*}
  Consequently, $\alpha T_B \not\leqfin \real(F_q)$ for all $\alpha\in\R$, 
  which ends the proof of part (b).
\end{proof}

\section{Numerical examples}
\label{sec:NumericalExamples}
In the following we discuss two numerical examples for the
two-dimensional sign-definite case to illustrate the theoretical
results developed in Theorems~\ref{thm:Shape1}--\ref{thm:Shape2}.
The algorithm suggested below is preliminary, and does not immediately
extend to the indefinite case considered in Theorem~\ref{thm:Shape3}.

We assume that far field observations $\uinfty(\xhat_l;\theta_m)$
are available for $N$ equidistant observation and incident directions
\begin{equation}
  \label{eq:ObservationGrid}
  \xhat_l,\theta_m 
  \in \{(\cos\phi_n,\sin\phi_n)\in S^1 \;|\;
  \phi_n=(n-1)2\pi/N \,,\; n=0,\ldots,N-1 \} \,, 
\end{equation}
$1\leq l,m\leq N$.
Accordingly, the matrix
\begin{equation}
  \label{eq:FMatrix}
  \bfF_q 
  \,=\, \frac{2\pi}N [\uinfty(\xhat_l;\theta_m)]_{1\leq l,m\leq N} 
  \in \C^{N\times N} \,,
\end{equation}
approximates the far field operator $F_q$ from
\eqref{eq:FarfieldOperator}. 
If the support of the contrast function $q$, i.e., of the scattering
objects, is contained in the ball $\BR$ for some~$R>0$, then it is
appropriate to choose 
\begin{equation}
  \label{eq:SamplingCondition}
  N \,\gtrsim\, 2 kR \,.
\end{equation}
where as before $k$ denotes the wave number, to fully resolve the
relevant information contained in the far field patterns (see, e.g.,
\cite{griesmaier2017uncertainty1}). 

We consider an equidistant grid on the region of interest
\begin{equation}
  \label{eq:Grid}
  [-R,R]^2 \,=\, \bigcup_{j=1}^JP_j \,, \qquad R>0 \,,
\end{equation}
with quadratic pixels $P_j = z_j + [-\frac{h}2,\frac{h}2]^2$,
$1\leq j\leq J$, where $z_j\in\R^2$ denotes the center of $P_j$ and
$h$ is its side length.
In this case a short computation shows that for each pixel $P_j$ the
operator $T_{P_j}$ from \eqref{eq:DefTB1} is approximated by the
matrix 
\begin{equation*}
  \bfT_{P_j}
  = \frac{2\pi}N 
  \Bigl[\! (kh)^2
  e^{\rmi k z_j\cdot(\theta_m-\theta_l)}
  \sinc\Bigl(\frac{kh}{2}(\theta_m-\theta_l)_1\Bigr)
  \sinc\Bigl(\frac{kh}{2}(\theta_m-\theta_l)_2\Bigr)
  \!\Bigr]_{1\leq l,m\leq N} 
  \!\in \C^{N\times N} .
\end{equation*}
Therewith, we compute the eigenvalues
$\lambda_1^{(j)},\ldots,\lambda_N^{(j)}\in\R$ of the self-adjoint
matrix 
\begin{equation}
  \label{eq:AMatrix}
  \bfA_{P_j}
  \,=\, \sign(q) (\real(\bfF_q) - \alpha \bfT_{P_j}) \,, 
  \qquad 1\leq j\leq J \,.
\end{equation}

For numerical stabilization, we discard those eigenvalues whose
absolute values are smaller than some threshold.
This number depends on the quality of the data.
If there are good reasons to believe that $\bfA_{P_j}$ is known up to
a perturbation of size~$\delta>0$ (with respect to the spectral norm),
then we can only trust in those eigenvalues with magnitude larger than
$\delta$ (see, e.g., \cite[Thm.~7.2.2]{GolVan96}).
To obtain a reasonable estimate for $\delta$, we use the
magnitude of the non-unitary part of 
$\bfS_q := (\bfI_N+\rmi/(4\pi)\bfF_q)$, i.e. we take 
$\delta = \|\bfS_q^*\bfS_q-\bfI_N\|_2$, since this quantity should be
zero for exact data and be of the order of the data error, otherwise.

Assuming that the contrast function $q$ is either larger or 
smaller than zero a.e.\ in $\supp(q)$, and that the parameter
$\alpha\in\R$ satisfies the conditions in part (a) of
Theorems~\ref{thm:Shape1} or~\ref{thm:Shape2}, respectively, we then
simply count for each pixel $P_j$ the number of negative eigenvalues
of $\bfA_{P_j}$, and define the \emph{indicator function} 
$I_\alpha:[-R,R]^2\to\N$,
\begin{equation}
  \label{eq:Itilde}
  I_\alpha(x) 
  \,=\, \# \{\lambda_n^{(j)} \;|\; 
  \lambda_n^{(j)}<-\delta \,,\; 1\leq n\leq N \} \,,
  \qquad \text{if } x\in P_j \,.
\end{equation}
Theorems~\ref{thm:Shape1}--\ref{thm:Shape2} suggest that $I_\alpha$ is
larger on pixels $P_j$ that do not intersect the support $\supp(q)$ of
the scattering object than on pixels $P_j$ contained in $\supp(q)$. 

\begin{example}
\label{exa:1}
  \begin{figure}[t]
    \centering
    \includegraphics[height=3.4cm]{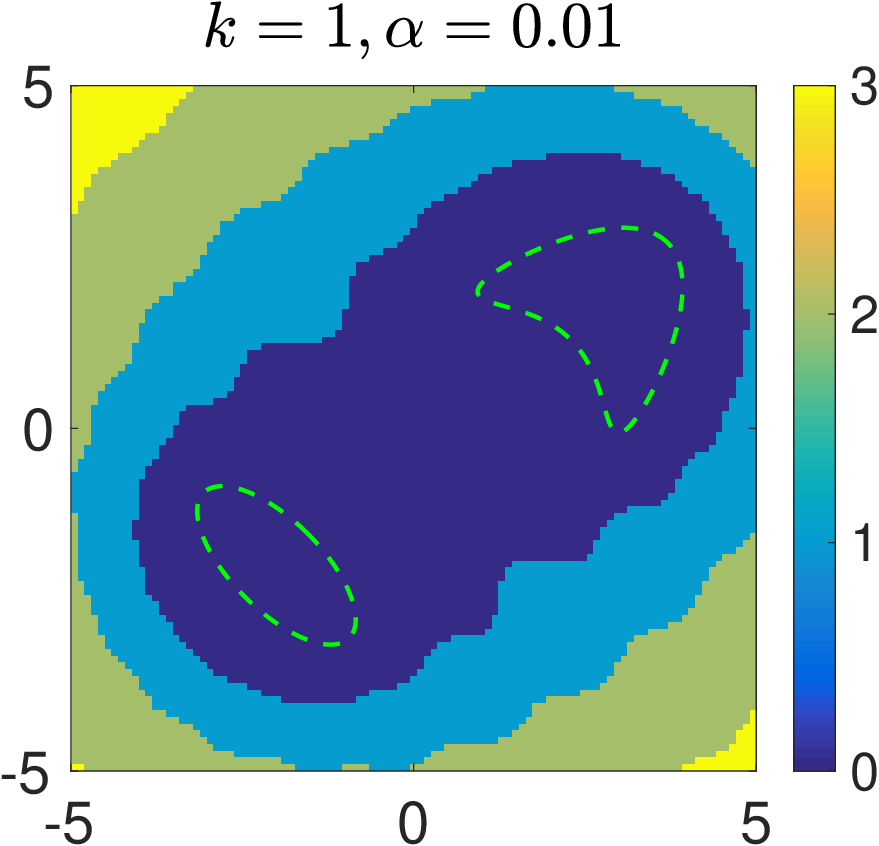}
    \hspace*{1.0em}
    \includegraphics[height=3.4cm]{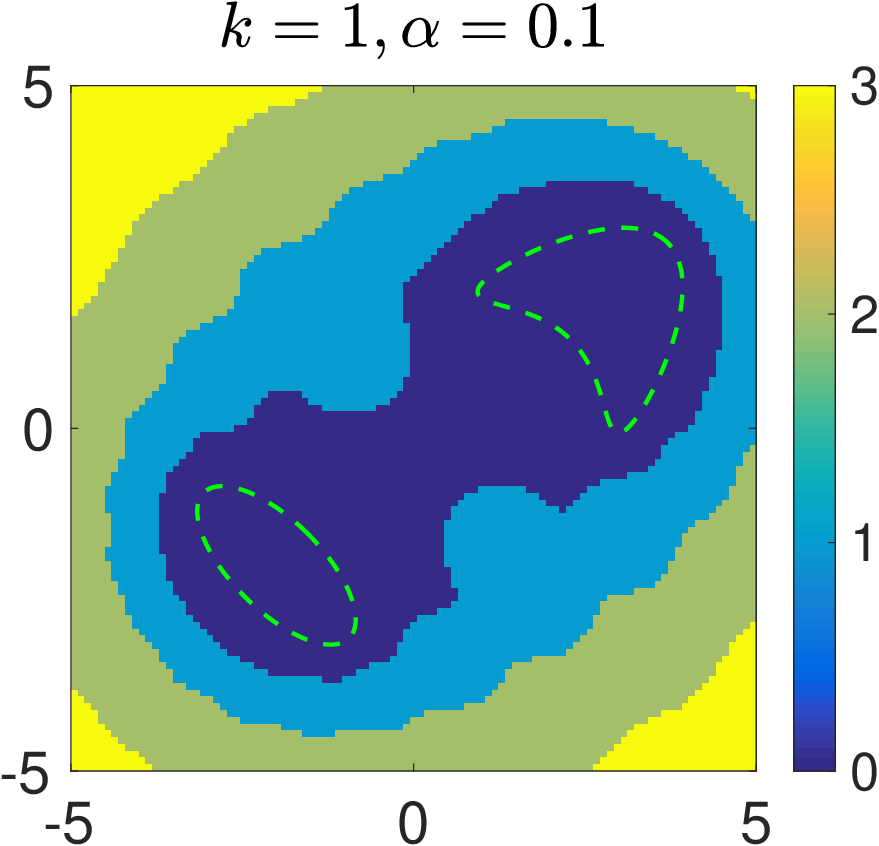}
    \hspace*{1.0em}
    \includegraphics[height=3.4cm]{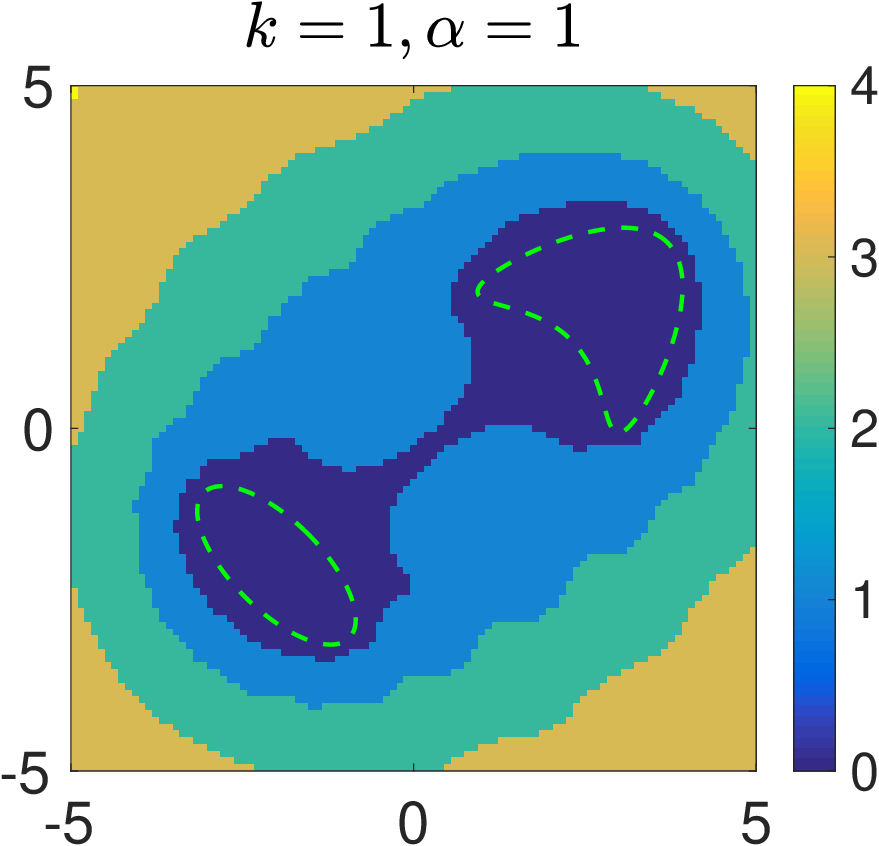}\\[0.5em]
    \includegraphics[height=3.4cm]{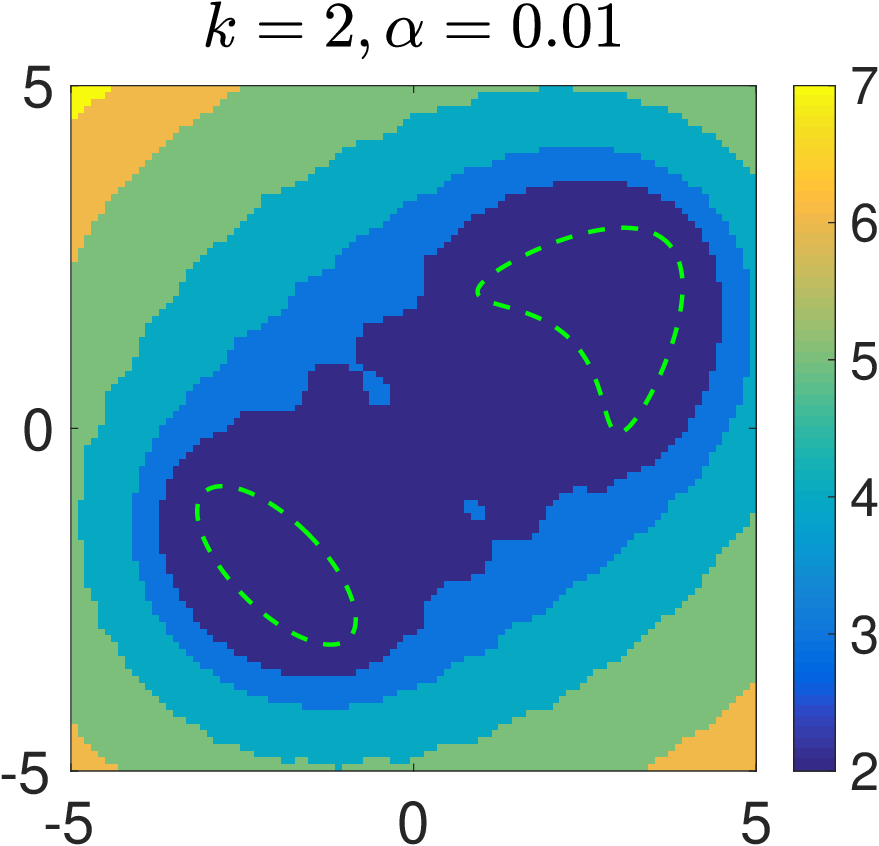}
    \hspace*{1.0em}
    \includegraphics[height=3.4cm]{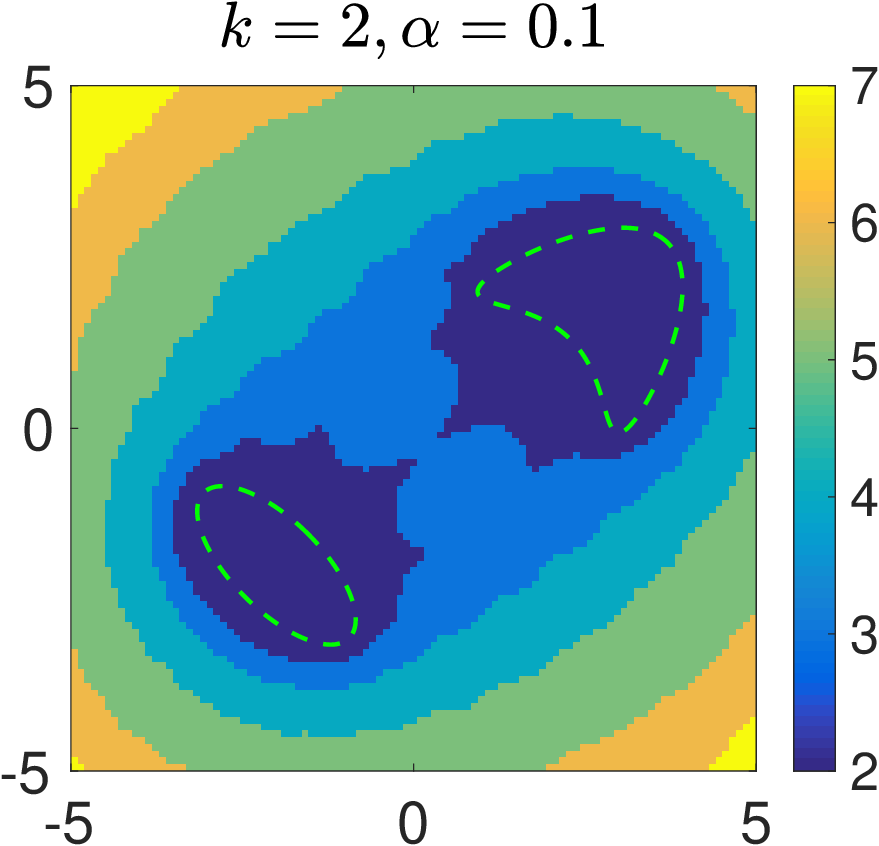}
    \hspace*{1.0em}
    \includegraphics[height=3.4cm]{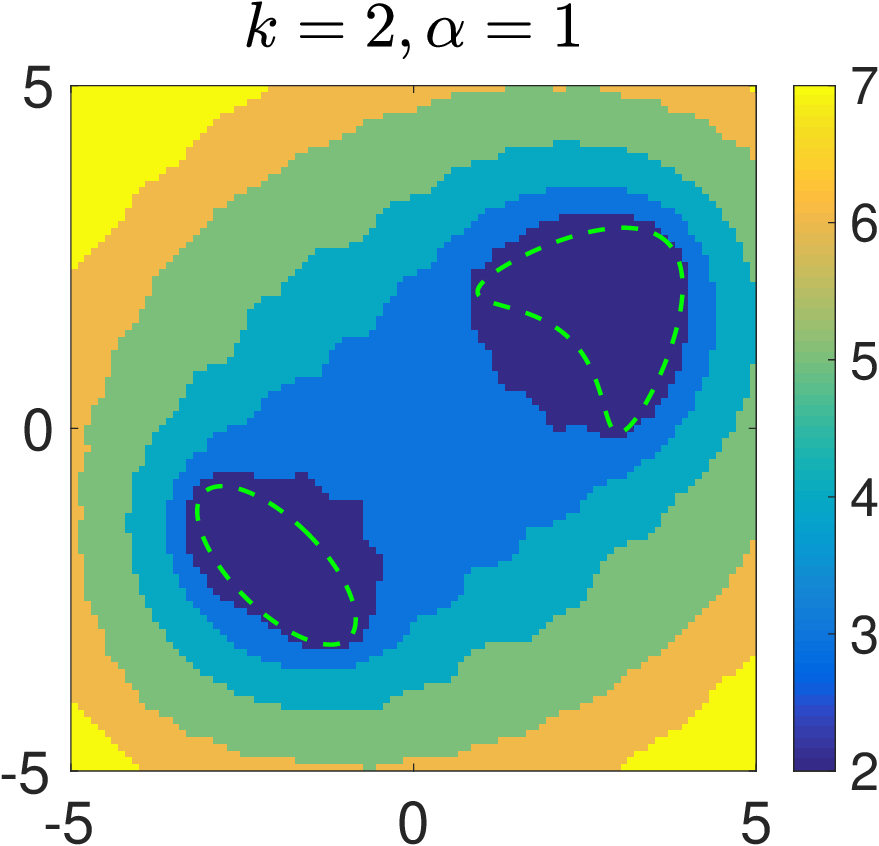}\\[0.5em]
    \hspace*{-0.1em}
    \includegraphics[height=3.4cm]{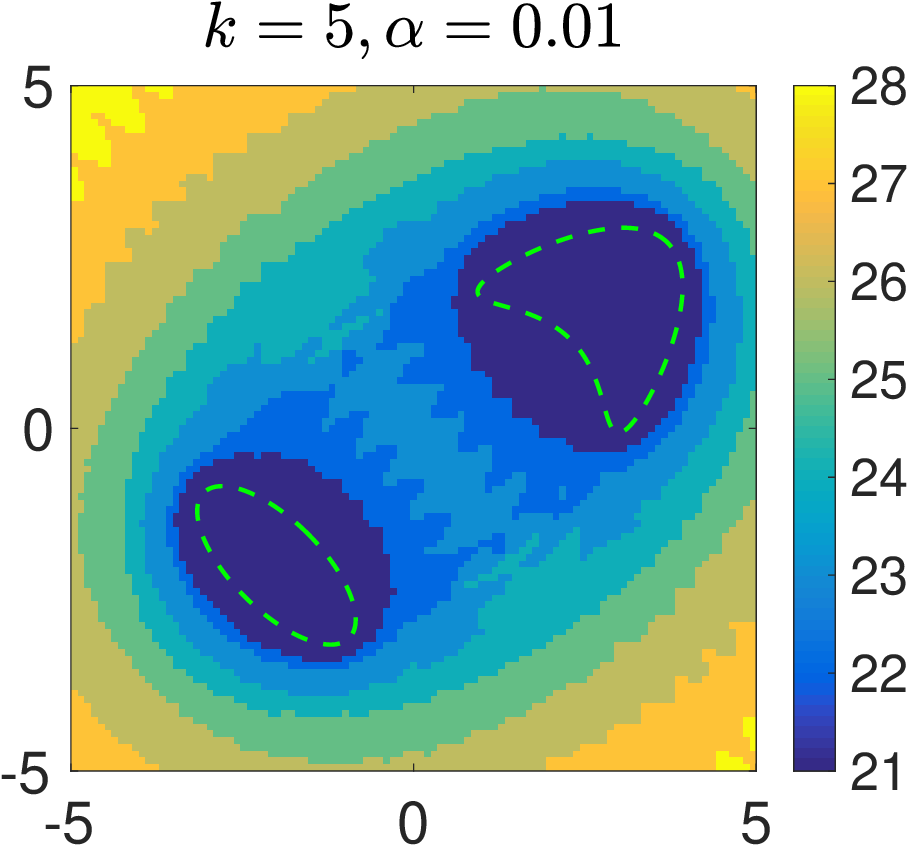}
    \hspace*{0.7em}
    \includegraphics[height=3.4cm]{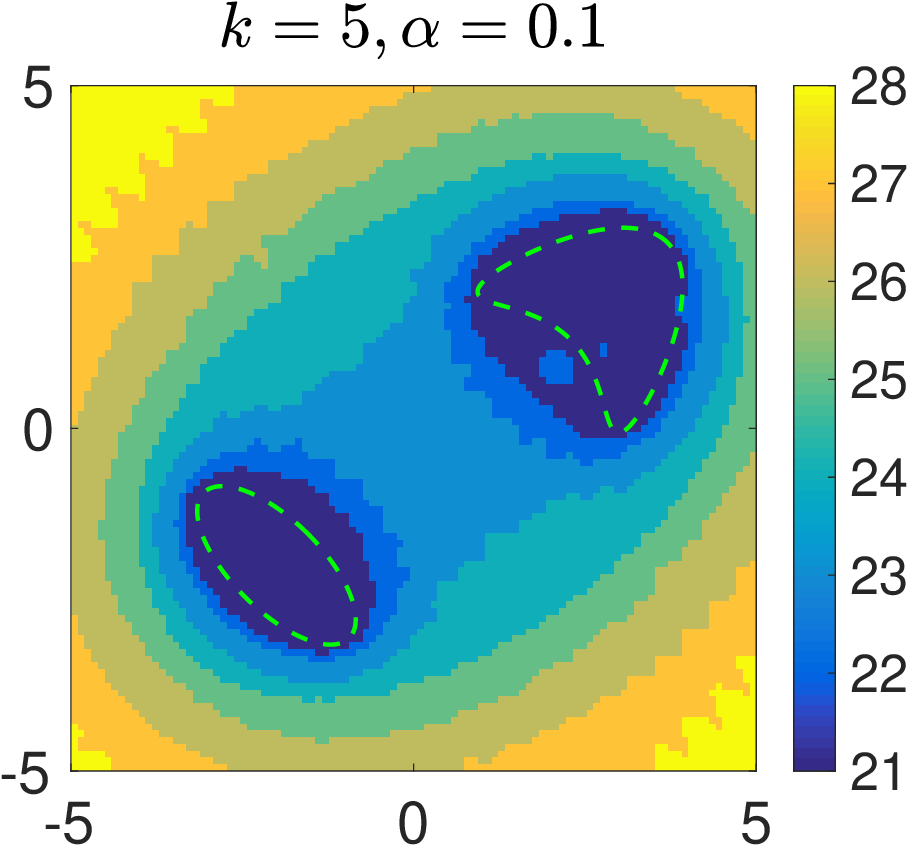}
    \hspace*{0.7em}
    \includegraphics[height=3.4cm]{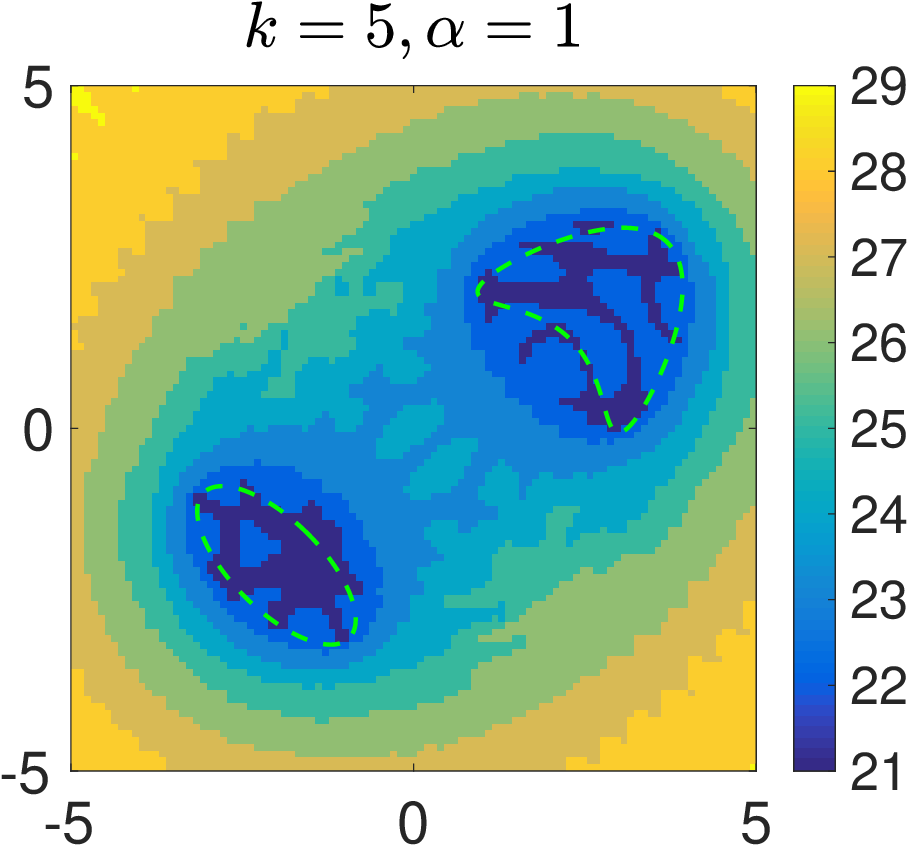}
    \caption{\small Scatterers with positive contrast functions: 
      Visualization of the indicator function $I_\alpha$
      for three different parameters
      $\alpha=0.01,0.1$, and $\alpha=1$ (left to right) and three different
      wave numbers $k=1,2$, and $k=5$ (top down).  Exact shape of the
      scatterers is shown as dashed lines.
    } 
    \label{fig:Exa1-1}
  \end{figure}

  We consider two penetrable scatterers, a kite and an ellipse, with
  positive constant contrast functions $q=1$ (kite) and $q=2$
  (ellipse) as sketched in Figure~\ref{fig:Exa1-1} (dashed lines), and
  simulate the corresponding far field matrix
  $\bfF_q\in\C^{64\times64}$ for $N=64$ observation and incident
  directions as in \eqref{eq:ObservationGrid} using a Nystr\"om method
  for a boundary integral formulation of the scattering problem with
  three different wave numbers $k=1,2$, and $k=5$. 

  In Figure~\ref{fig:Exa1-1}, we show color coded plots of the
  indicator function $I_\alpha$ from \eqref{eq:Itilde} with threshold
  parameter $\delta=10^{-14}$ (i.e., the number of negative
  eigenvalues smaller than $-\delta=-10^{-14}$ of the matrix
  $\bfA_{P_j}$ from 
  \eqref{eq:AMatrix} on each pixel $P_j$) in the region of interest
  ${[-5,5]^2\tm\R^2}$ for three different parameters
  $\alpha=0.01,0.1$, and $\alpha=1$ (left to right) and three different
  wave numbers $k=1,2$, and $k=5$ (top down).
  The equidistant rectangular sampling grid on the region of interest
  from~\eqref{eq:Grid} consists of~$100$ pixels in each direction.

  Overall, the number of negative eigenvalues of the matrix
  $\bfA_{P_j}$ increases with increasing wave number, and it is larger
  on pixels $P_j$ sufficiently far away from the support of the
  scatterers than on pixels $P_j$ inside, as suggested by
  Theorems~\ref{thm:Shape1}--\ref{thm:Shape2}.
  The lower value always coincides with the number of negative
  eigenvalues of the real part $\real(\bfF_q)$ of the far field matrix
  from \eqref{eq:FMatrix} that are smaller than the threshold~$-\delta$. 
  The number of eigenvalues of $\bfA_{P_j}$, $j=1,\ldots,J$, whose
  absolute values are larger than~$\delta$ is approximately (on
  average) $25$ (for $k=1$), $36$ (for $k=2$), and $62$ (for $k=5$),
  independent of $\alpha$. 

  If the parameter $\alpha$ is suitably chosen, depending on the wave
  number, then the lowest level set of the indicator function
  $I_\alpha$ nicely approximates the support of the two scatterers.
\hfill$\lozenge$
\end{example}

\begin{example}
\label{exa:2}
  \begin{figure}[t]
    \centering
    \includegraphics[height=3.4cm]{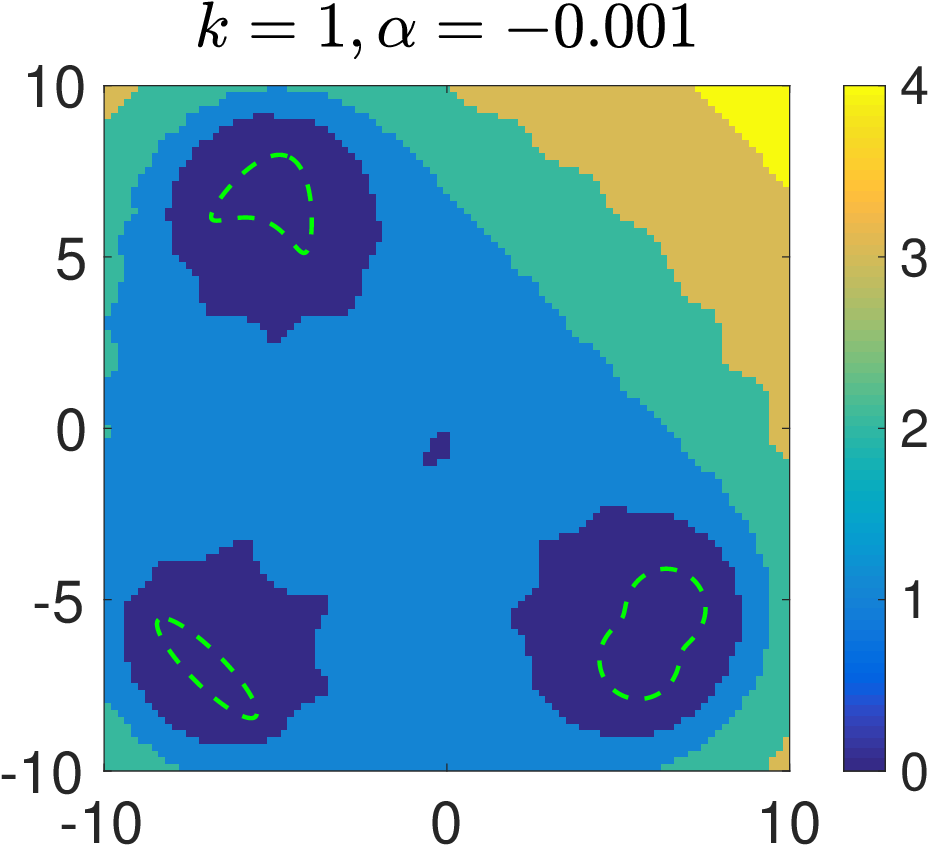}
    \hspace*{1.0em}
    \includegraphics[height=3.4cm]{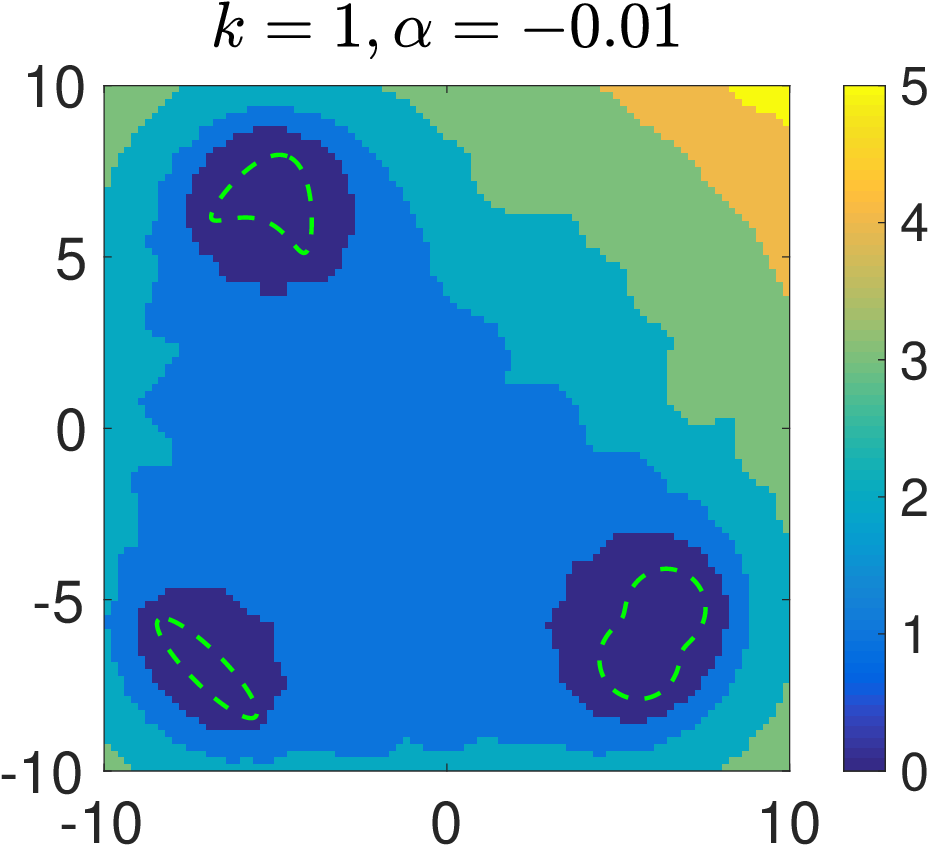}
    \hspace*{1.0em}
    \includegraphics[height=3.4cm]{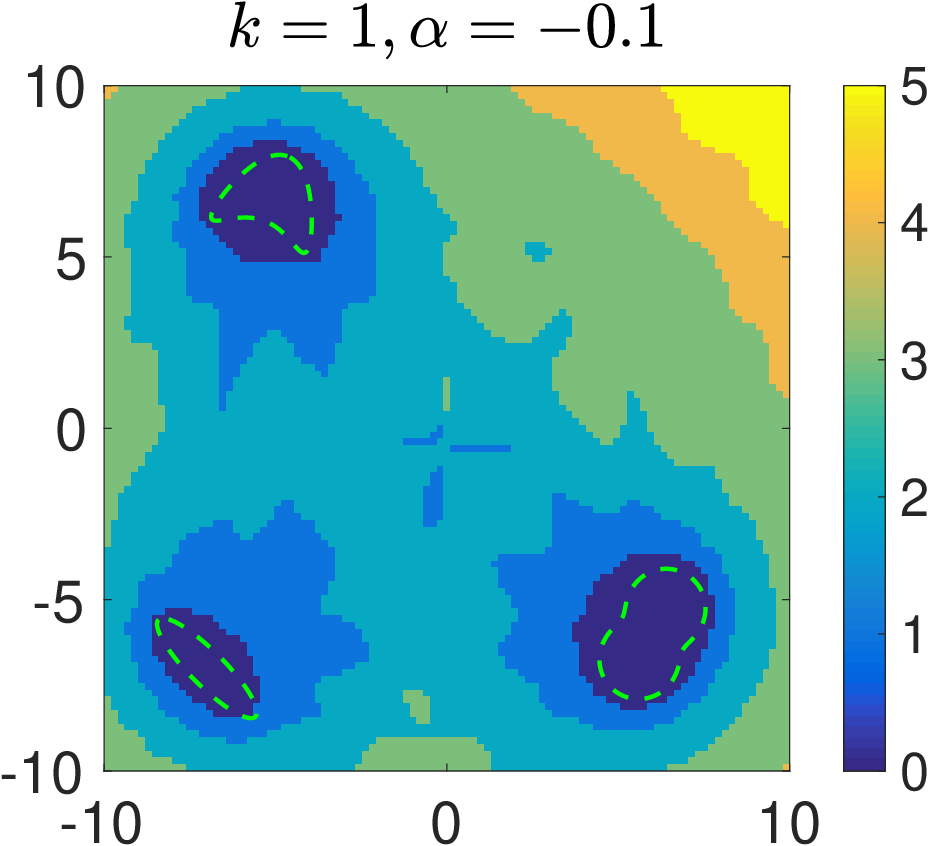}\\[0.5em]
    \includegraphics[height=3.4cm]{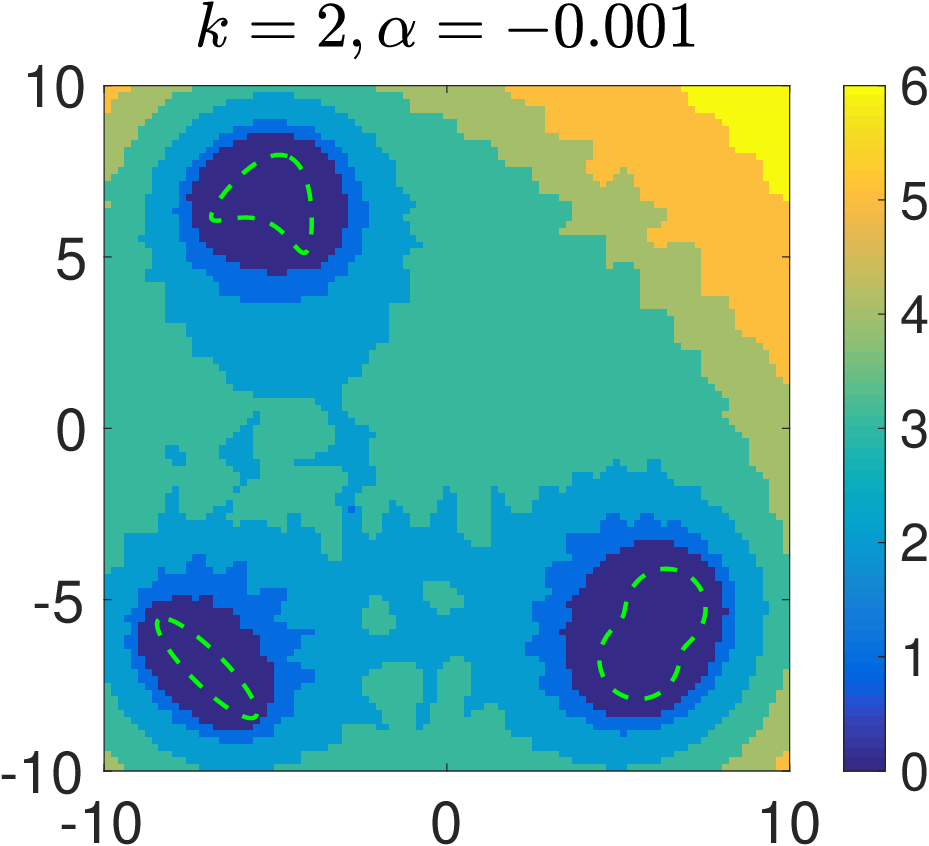}
    \hspace*{1.0em}
    \includegraphics[height=3.4cm]{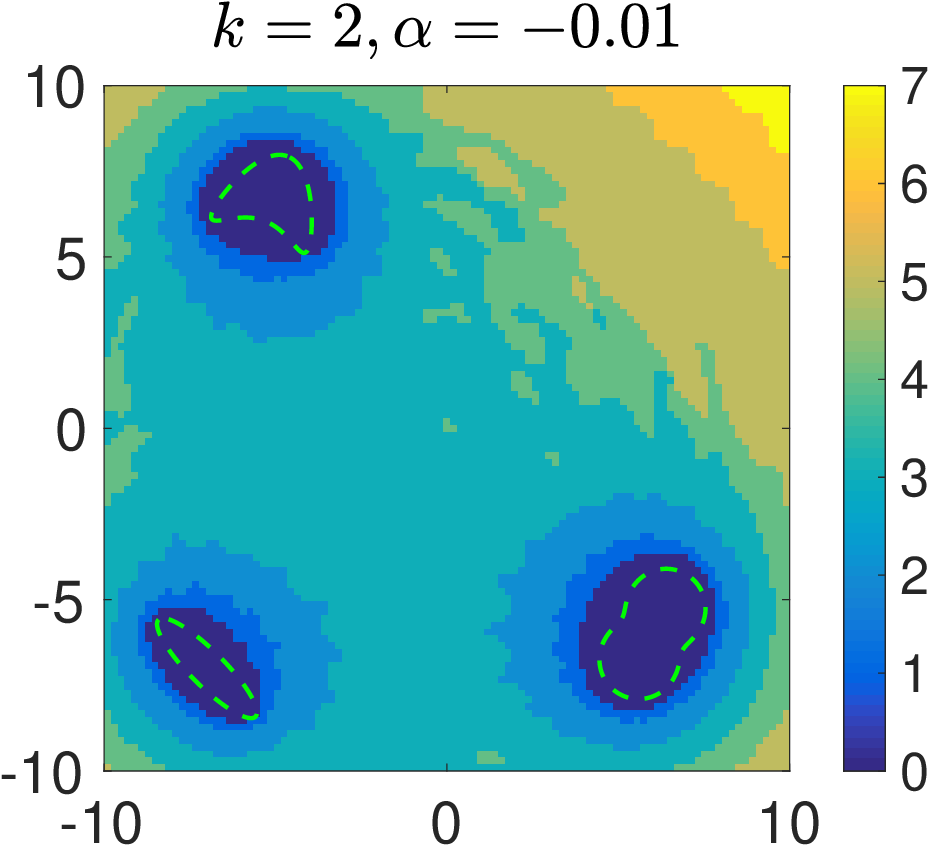}
    \hspace*{1.0em}
    \includegraphics[height=3.4cm]{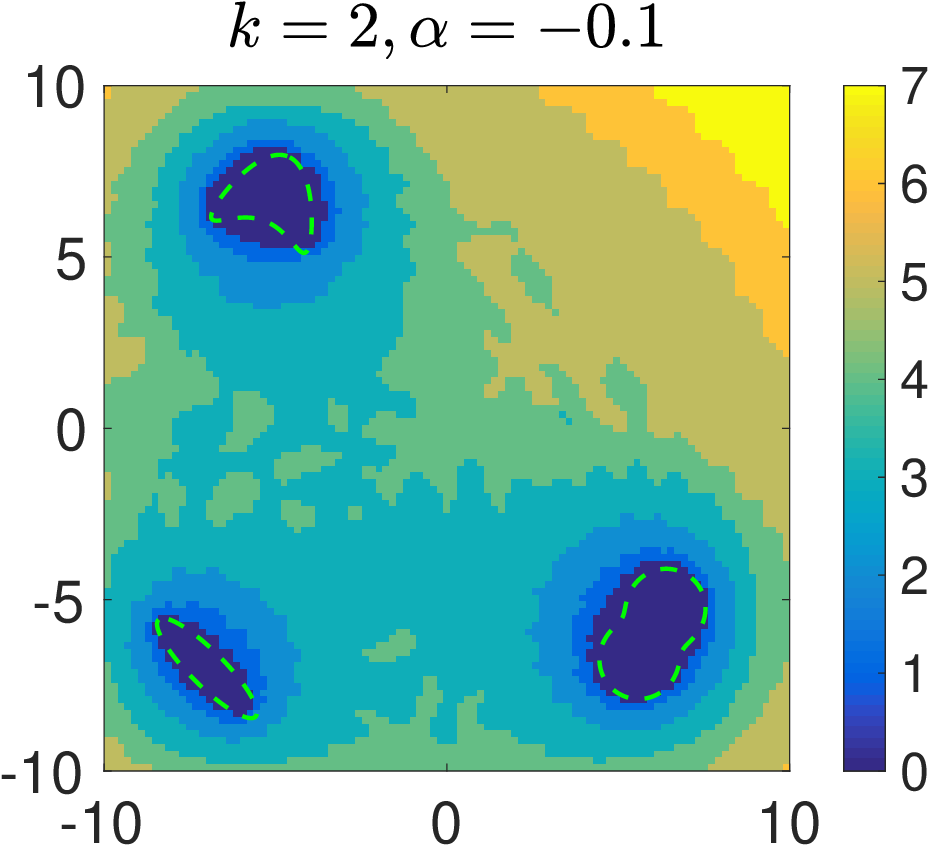}\\[0.5em]
    \hspace*{-0.1em}
    \includegraphics[height=3.4cm]{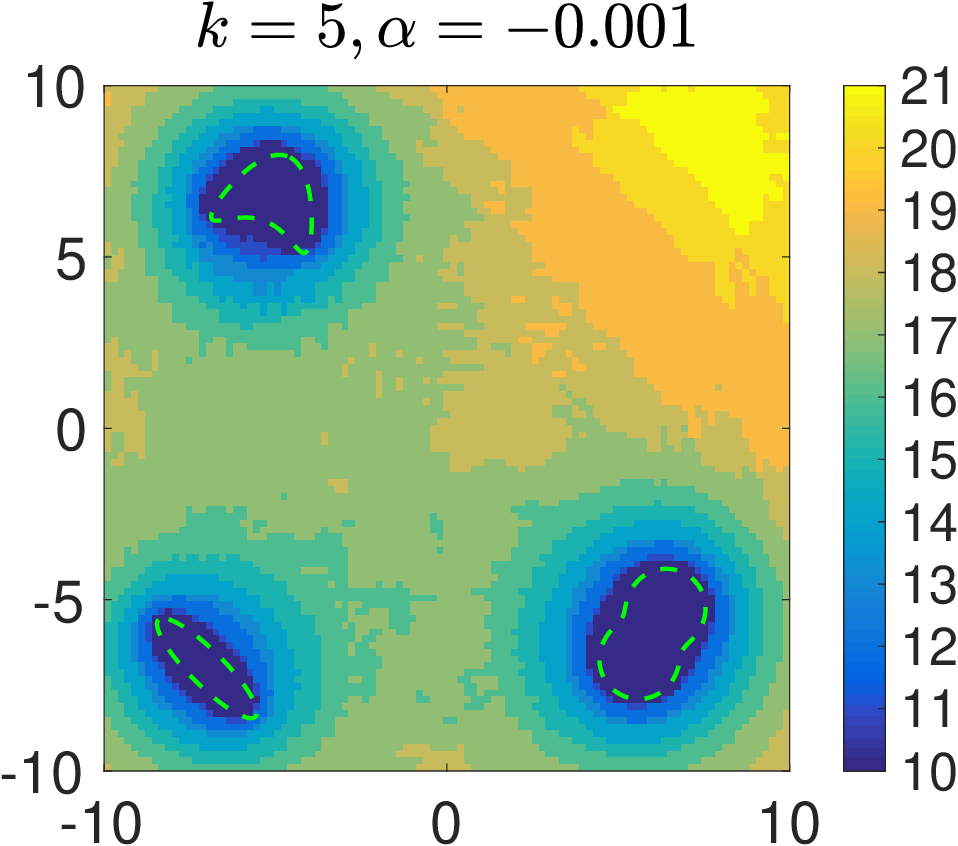}
    \hspace*{0.7em}
    \includegraphics[height=3.4cm]{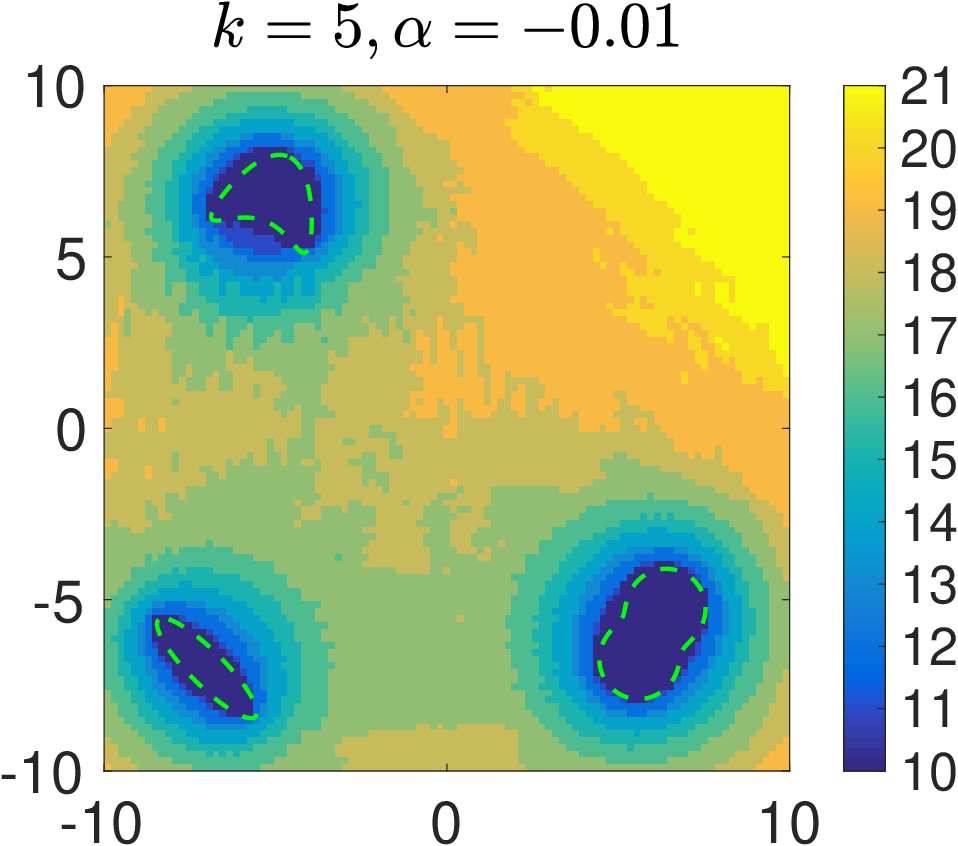}
    \hspace*{0.7em}
    \includegraphics[height=3.4cm]{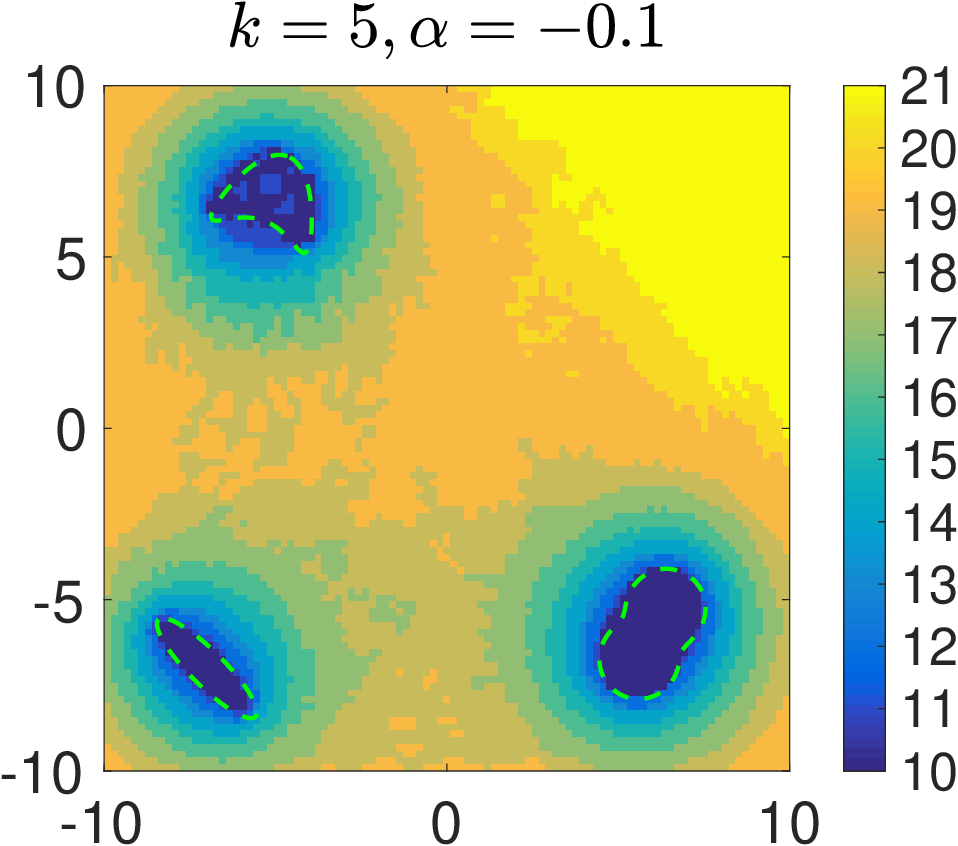}
    \caption{\small Scatterers with negative contrast functions: 
      Visualization of the indicator function $I_\alpha$
      for three different parameters $\alpha=-0.001,-0.01$, and
      $\alpha=-0.1$ (left to right) and three different wave numbers
      $k=1,2$, and $k=5$ (top down).  Exact shape of the scatterers is
      shown as dashed lines.
    } 
    \label{fig:Exa2-1}
  \end{figure}
  
  In the second example, we consider three penetrable scatterers, a
  kite, an ellipse and a nut-shaped scatterer, with negative
  constant contrasts $q=-0.8$ (kite), $q=-0.4$ (nut), and $q=-0.2$
  (ellipse) as sketched in Figure~\ref{fig:Exa2-1} (dashed lines), and
  simulate the corresponding far field matrix
  $\bfF_q\in\C^{128\times128}$ for $N=128$ observation and incident
  directions for three different wave numbers $k=1,2$, and $k=5$. 
  We increase the number of discretization points because the diameter
  of the support of this configuration of scattering objects is
  roughly twice as large as in the previous example (i.e., to fulfill
  the sampling condition \eqref{eq:SamplingCondition}). 

  In Figure~\ref{fig:Exa1-1}, we show color coded plots of the
  indicator function $I_\alpha$ from \eqref{eq:Itilde} with threshold
  parameter $\delta=10^{-14}$ in the region of interest 
  ${[-10,10]^2\tm\R^2}$ for three different parameters
  $\alpha=-0.001,-0.01$, and $\alpha=-0.1$ (left to right) and three
  different wave numbers $k=1,2$, and $k=5$ (top down).
  The equidistant rectangular sampling grid on the region of interest
  from~\eqref{eq:Grid} on this region of interest consists of $100$
  pixels in each direction.

  Again, the number of negative eigenvalues of the matrix $\bfA_{P_j}$
  increases with increasing wave number, and it is larger on pixels
  $P_j$ sufficiently far away from the support of the scatterers than
  on pixels $P_j$ inside, in compliance with
  Theorems~\ref{thm:Shape1}--\ref{thm:Shape2}. 
  The lower value always coincides with the number of positive
  eigenvalues of the matrix $\real(\bfF_q)$ from \eqref{eq:FMatrix}
  that are larger than the threshold $\delta=10^{-14}$.
  The number of eigenvalues of $\bfA_{P_j}$, $j=1,\ldots,J$, whose
  absolute values are larger than $\delta$ is
  approximately (on average) $39$ (for $k=1$), $60$ (for $k=2$), and
  $115$ (for $k=5$), independent of $\alpha$.

  If the parameter $\alpha$ is suitably chosen, depending on the wave
  number, then the support of the indicator function $I_\alpha$
  approximates the support of the three scatterers rather well.
\hfill$\lozenge$
\end{example}

An efficient and suitably regularized numerical implementation of the
theoretical results developed in
Theorems~\ref{thm:Shape1}--\ref{thm:Shape3} is beyond the scope of
this article, and the preliminary algorithm discussed in this section
cannot be considered competitive when compared against
state-of-the-art implementations of linear sampling or factorization
methods. 
The numerical results in the Examples~\ref{exa:1}--\ref{exa:2} have
been obtained for highly accurate simulated far field data.
Further numerical tests showed that the algorithm is rather
sensitive to noise in the data.

\section*{Conclusions}
We have derived new monotonicity relations for the far field operator
for the inverse medium scattering problem with compactly supported
scattering objects, and we used them to provide novel monotonicity
tests to determine the support of unknown scattering objects from
far field observations of scattered waves corresponding to infinitely
many plane wave incident fields. 
Along the way we have shown the existence of localized wave functions
that have arbitrarily large norm in some prescribed region while
having arbitrarily small norm in some other prescribed region. 

When compared to traditional qualitative reconstructions methods,
advantages of these new characterizations are that they apply to 
indefinite scattering configurations. 
Moreover, these characterizations are independent of transmission
eigenvalues.
However, although we presented some preliminary numerical examples for
the sign definite case, a stable numerical implementation of these
monotonicity tests still needs to be developed.

\section*{Acknowledgments}
This research was initiated at the Oberwolfach Workshop
``Computational Inverse Problems for Partial Differential Equations''
in May 2017 organized by Liliana Borcea, Thorsten Hohage, and Barbara
Kaltenbacher.  We thank the organizers and the Oberwolfach Research
Institute for Mathematics~(MFO) for the kind invitation.

  \bibliographystyle{abbrv}
  \bibliography{literaturliste}



\newpage

\thispagestyle{empty}
\phantom{empty page}

\newpage

\thispagestyle{empty}

\begin{center}
{\bfseries\MakeUppercase{Erratum: Monotonicity in inverse medium scattering on unbounded domains}}\\[+2ex]

\footnotesize
\MakeUppercase{Roland Griesmaier}\footnote[1]{Institut f\"ur Angewandte und Numerische Mathematik, Karlsruher Institut f\"ur Technologie, 76049 Karlsruhe, Germany (\email{roland.griesmaier@kit.edu}).} 
and
\MakeUppercase{Bastian Harrach}\footnote[2]{Institut f\"ur Mathematik, Universit\"at Frankfurt, 60325 Frankfurt am Main, Germany (\email{harrach@math.uni-frankfurt.de}).}\\[+2ex]
\end{center}

\let\thefootnote\relax\footnotetext{\hrule \vspace{1ex} \centering This is a preprint version of an article
that is accepted for publication in \emph{SIAM J. Appl. Math.}
}

\setcounter{page}{1}
\setcounter{section}{0}

\begin{abstract}
  We correct a mistake in the proof of Theorem~5.3 in 
  [R.~Griesmaier and B.~Harrach. SIAM J. Appl. Math., 
  78(5):2533--2557, 2018].
\end{abstract}

\begin{keywords}
  Inverse scattering, Helmholtz equation, monotonicity,
  far field operator, inhomogeneous medium
\end{keywords}

\begin{AMS}
  35R30, 65N21
\end{AMS}

\section{An error in the proof of Theorem~5.3 in {[3]}}
\label{erratum:sec:err}
At the end of the proof of Theorem~5.3 in \protect\citeerratum{GriHar18} 
``Applying Theorem~4.5 with $D=\BR\setminus\ol{O}$, $q_1=0$, and
$q_2=q$~\ldots'' is not possible, because the assumption of
Theorem~4.5 in \protect\citeerratum{GriHar18} that
$q_1(x)=q_2(x)$ for a.e.~$x\in\Rd\setminus\ol{D}$ is not satisfied for
this choice of $D$, $q_1$ and $q_2$.

To fix this issue we will extend the results on localized wave
functions from Section~4 of \protect\citeerratum{GriHar18} in
Section~\ref{erratum:sec:localized} below.  
Then, in Section~\ref{erratum:sec:correction} we will reformulate Theorem~5.3
of \protect\citeerratum{GriHar18}, making stronger assumptions on the domains and on
the index of refraction, and we will correct the final argument in the
original proof in \protect\citeerratum{GriHar18}.

\section{Simultaneously localized wave functions}
\label{erratum:sec:localized}
We establish the existence of 
simultaneously localized wave functions that have arbitrarily
large norm on some prescribed region~$E\tm\Rd$ while at the same time
having arbitrarily small norm in a different region $M\tm\Rd$,
assuming among others that $\Rd\setminus(\ol{E}\cup\ol{M})$ is
connected. 
The result generalizes Theorem~4.1 in \protect\citeerratum{GriHar18} in the sense
that we not only control the total field but also the incident field.
Similar results have recently been established for the Schr\"odinger
equation in \citeerratum[Thm.~3.11]{harrach2020monotonicity} and for the
Helmholtz obstacle scattering problem in \citeerratum[Thm.~4.5]{AlbGri20}. 

\begin{theorem}
  \label{erratum:thm:LocPot1}
  Suppose that $q\in\LinftyCRd$, and let $E,M\tm\Rd$ be open and
  Lipschitz bounded such that $\supp(q)\tm\ol{E}\cup\ol{M}$, 
  $\Rd\setminus(\ol{E}\cup\ol{M})$ is connected, and 
  $E\cap M=\emptyset$. 
  Assume furthermore that there is a connected subset
  $\Gamma\tm\di E\setminus\ol{M}$ that is relatively open and
  $C^{1,1}$ smooth. 

  Then for any finite dimensional subspace
  $V\tm\LtSd$ there exists a sequence $(g_m)_{m\in\N}\tm\Vperp$ such
  that 
  \begin{equation*}
    \int_E |u_{q,g_m}|^2\dx \to \infty 
    \qquad\text{and}\qquad
    \int_M \bigl( |u_{q,g_m}|^2 + |\ui_{g_m}|^2 \bigr) \dx \to 0 
    \qquad \text{as } m\to\infty \,,
  \end{equation*}
  where $\ui_{g_m}, u_{q,g_m}\in \Heinsloc$ are given by
  {\rm (2.8a)--(2.8b)} in {\rm \protect\citeerratum{GriHar18}} with $g=g_m$.
\end{theorem}

The proof of Theorem~\ref{erratum:thm:LocPot1} relies on the following
three lemmas. 

\begin{lemma}
  \label{erratum:lmm:LocPot1}
  Suppose that $q\in\LinftyCRd$, let $n^2=1+q$, and assume that
  $D\tm\Rd$ is open and bounded.  We define  
  \begin{equation*}
    L_{q,D}:\LtSd\to H^1(D) \,, \quad g\mapsto u_{q,g}|_D \,,
  \end{equation*}
  where $u_{q,g}\in \Heinsloc$ is given by {\rm (2.8b)}
  in~{\rm\citeerratum{GriHar18}}. 
  Then $L_{q,D}$ is a linear operator and its adjoint is given by
  \begin{equation*}
    L_{q,D}^*: H^1(D)^*\to\LtSd \,, \quad f\mapsto \Scal_q^*w^\infty \,,
  \end{equation*}
  where $H^1(D)^*$ is the dual of $H^1(D)$, $\Scal_q^*$ denotes the
  adjoint of the scattering operator from {\rm (2.7)}
  in~{\rm\citeerratum{GriHar18}}, and ${w^\infty\in\LtSd}$ is the far field
  pattern of the radiating solution ${w\in\Heinsloc}$ to 
  \begin{equation}
    \label{erratum:eq:Defw}
    \Delta w+k^2n^2w \,=\, -f \qquad\text{in }\Rd \,.
  \end{equation}
\end{lemma}

\begin{proof}
  This follows from the same arguments that have been used in the
  proof of Lemma~4.2 in \protect\citeerratum{GriHar18}.
\end{proof}

\begin{lemma}
  \label{erratum:lmm:LocPot2}
  Suppose that $q\in\LinftyCRd$, and let $E,M\tm\Rd$ be open and
  Lipschitz bounded such that $\supp(q)\tm\ol{E}\cup\ol{M}$, 
  $\Rd\setminus(\ol{E}\cup\ol{M})$ is connected, and 
  $E\cap M=\emptyset$. 
  Assume furthermore that there is a connected subset
  $\Gamma\tm\di E\setminus\ol{M}$ that is relatively open and
  $C^{1,1}$ smooth. 
  Then,
  \begin{equation*}
    \Rcal(L_{q,E}^*)\not\tm \Rcal\bigl(
    \begin{pmatrix}
      L_{q,M}^* & L_{0,M}^*
    \end{pmatrix}\bigr)
  \end{equation*}
  and there exists an infinite dimensional subspace
  $Z\tm\Rcal(L_{q,E}^*)$ such that
  \begin{equation*}
    Z \cap \Rcal\bigl(
    \begin{pmatrix}
      L_{q,M}^* & L_{0,M}^*
    \end{pmatrix}\bigr) \,=\, \{0\} \,.
  \end{equation*}
\end{lemma}

\begin{proof}
  Let $h\in \Rcal(L_{q,E}^*)\cap \Rcal\bigl(
  \begin{pmatrix}
    L_{q,M}^* & L_{0,M}^*
  \end{pmatrix}\bigr)$.  
  Then
  Lemma~\ref{erratum:lmm:LocPot1} shows that there exist $f_{q,E}\in H^1(E)^*$ and
  $f_{q,M}, f_{0,M}\in H^1(M)^*$ such that the far field 
  patterns $w^\infty_{q,E}, w^\infty_{q,M}, w^\infty_{0,M}$ of the
  radiating solutions $w_{q,E},w_{q,M},w_{0,M}\in\Heinsloc$ to 
  \begin{align*}
    \Delta w_{q,E}+k^2(1+q)w_{q,E} &\,=\, -f_{q,E} &&\text{in } \Rd \,,\\
    \Delta w_{q,M}+k^2(1+q)w_{q,M} &\,=\, -f_{q,M} &&\text{in } \Rd \,,\\
    \Delta w_{0,M}+k^2w_{0,M} &\,=\, -f_{0,M} &&\text{in } \Rd \,,
  \end{align*}
  satisfy
  \begin{equation*}
    h 
    \,=\, \Scal_q^* w^\infty_{q,E} 
    \,=\, w^\infty_{0,M} + \Scal_q^* w^\infty_{q,M} 
    \,.
  \end{equation*}
  Here we used that $\Scal_0$ is the identity operator.
  Accordingly, using the definition of the scattering operator in
  (2.7) of \protect\citeerratum{GriHar18}, we find that
  \begin{equation*}
    \begin{split}
      0
      &\,=\, w^\infty_{q,E} - w^\infty_{q,M} - \Scal_q w^\infty_{0,M}\\ 
      &\,=\, w^\infty_{q,E} - w^\infty_{q,M} 
      - w^\infty_{0,M} - 2\rmi k|C_d|^2F_q w^\infty_{0,M}\\ 
      &\,=\, w^\infty_{q,E} - (w^\infty_{q,M} 
      + w^\infty_{0,M} + v^\infty_{q}) \,,
    \end{split}
  \end{equation*}
  where $v^\infty_{q}$ is the far field of a radiating solution 
  $v_{q}\in\Heinsloc$ to
  \begin{equation*}
    \Delta v_{q}+k^2(1+q)v_{q} \,=\, 0 \qquad\text{in } \Rd \,.
  \end{equation*}
  Since $\supp(q)\tm\ol{E}\cup\ol{M}$ and
  $\Rd\setminus(\ol{E}\cup\ol{M})$ is connected, 
  Rellich's lemma and unique continuation guarantee that
  \begin{equation}
    \label{erratum:eq:IdentityWs}
    w_{q,E} - (w_{q,M} + w_{0,M} + v_{q}) \,=\, 0 \qquad 
    \text{in } \Rd\setminus(\ol{E}\cup\ol{M})
  \end{equation}
  (cf., e.g., \citeerratum[Thm.~2.14]{erratumcolton2013inverse}).
  
  Next we discuss the regularity of the traces of $w_{q,E}$ and
  $w_{q,M}+w_{0,M}+v_{q}$ at the boundary segment 
  $\Gamma \tm \di E\setminus\ol{M}$.
  W.l.o.g.\ we may assume that $\Gamma$ is bounded away from~$\ol{M}$. 
  Since $\supp(f_{q,M}+f_{0,M})\tm\ol{M}$, interior regularity results
  (see, e.g., \citeerratum[Thm.~4.18]{McL00}) show that 
  $(w_{q,M}+w_{0,M}+v_{q})\big|_{\Gamma}\in H^{\frac32}(\Gamma)$.
  Thus \eqref{erratum:eq:IdentityWs} implies that 
  $w_{q,E}\big|_{\Gamma}^+\in H^{\frac32}(\Gamma)$ as well.

  On the other hand, let $\widetilde{H}^{\frac12}(\Gamma)$ be the
  closure of $\Dcal(\Gamma)$ in $H^{\frac12}(\Gamma)$ (see, e.g.,
  \citeerratum[p.~99]{McL00}). 
  We will construct sources $f\in H^1(E)^*$ such that 
  $L_{q,E}^*f \not\in \Rcal\bigl(
    \begin{pmatrix}
      L_{q,M}^* & L_{0,M}^*
    \end{pmatrix}\bigr)$.
  Given any $g \in \widetilde{H}^{\frac12}(\Gamma)$, we denote by 
  $\widetilde{g} \in H^{\frac12}(\di E)$ its extension to $\di E$ by
  zero. 
  Accordingly, let $u^+\in H^1_\loc(\Rd\setminus\ol{E})$ be the
  radiating solution to the exterior Dirichlet problem
  \begin{equation}
    \label{erratum:eq:ExtDirichlet}
    \Delta u^+ + k^2n^2u^+ 
    \,=\, 0 \quad \text{in $\Rd\setminus\ol{E}$} \,, \qquad
    u^+ \,=\, \widetilde{g} \quad \text{ on $\di E$} \,.
  \end{equation}
  Similarly, we define $u^-\in H^1(E)$ as the solution to the interior
  Dirichlet problem
  \begin{equation*}
    \Delta u^- 
    \,=\, 0 \quad \text{in $E$} \,, \qquad
    u^- \,=\, \widetilde{g} \quad \text{ on $\di E$} \,.
  \end{equation*}
  Therewith we introduce $u\in L^2_\loc(\Rd)$ by
  \begin{equation*}
    u \,:=\,
    \begin{cases}
      u^- &\text{in } E \,,\\
      u^+ &\text{in } \Rd\setminus\ol{E} \,,
    \end{cases}
  \end{equation*}
  and $f\in H^1(E)^*$ by
  \begin{equation*}
    f 
    \,:=\, - k^2n^2u^- 
    - \gamma^*\Bigl( \frac{\di u}{\di\nu}\Big|_{\di E}^+ 
    - \frac{\di u}{\di\nu}\Big|_{\di E}^- \Bigr) \,,
  \end{equation*}
  where $\gamma^*: H^{-\frac12}(\di E)\to H^1(E)^*$ denotes the
  adjoint of the interior trace operator 
  $\gamma: H^1(E)\to H^{\frac12}(\di E)$.
  Then $u\in H^1_\loc(\Rd)$ (see, e.g., \citeerratum[Lmm.~5.3]{Mon03}), and
  \begin{equation*}
    \Delta u + k^2n^2u
    \,=\, - f \quad \text{in $\Rd$}
  \end{equation*}
  (see, e.g., \citeerratum[Lmm.~6.9]{McL00}).
  Accordingly, $L_{q,E}^* f = \Scal_q^* u^\infty$, where 
  $u^\infty\in \Lt(\Sd)$ coincides with the far field of the radiating
  solution $u^+$ to the exterior Dirichlet
  problem~\eqref{erratum:eq:ExtDirichlet}. 
  If $\widetilde{g} \not\in H^{\frac32}(\di E)$, then our regularity
  considerations above show that
  $L_{q,E}^*f \not\in \Rcal\bigl(
    \begin{pmatrix}
      L_{q,M}^* & L_{0,M}^*
    \end{pmatrix}\bigr)$.
  
  Now let $X\tm \widetilde{H}^{\frac12}(\Gamma)$ be an infinite
  dimensional subspace of $\widetilde{H}^{\frac12}(\Gamma)$ such that 
  $X\cap H^{\frac32}(\Gamma) = \{0\}$ (e.g., the subspace
  of piecewise linear functions on $\Gamma$ that vanish on~$\di\Gamma$
  as considered in the proof of Lemma~4.6 in \protect\citeerratum{AlbGri20}).
  Let $G_E: H^{\frac12}(\Gamma)\to L^2(\Sd)$ be the operator that maps
  $g\in H^{\frac12}(\Gamma)$ to the far field pattern of the radiating
  solution~$u^+$ of \eqref{erratum:eq:ExtDirichlet}, where 
  $\widetilde{g} \in H^{\frac12}(\di E)$ is again the extension of $g$
  to $\di E$ by zero.
  Then $G_E$ is one-to-one (see, e.g., \citeerratum[Thm.~3.2]{AlbGri20}), and
  thus $Z:=\Scal_q^* G_E(X)$ is infinite dimensional. 
  Furthermore, we have just shown that
  \begin{equation*}
    Z \tm \Rcal(L_{q,E}^*) 
    \qquad\text{and}\qquad
    Z \cap \Rcal\bigl(
    \begin{pmatrix}
      L_{q,M}^* & L_{0,M}^*
    \end{pmatrix}\bigr) \,=\, \{0\} \,.
  \end{equation*}
  \mbox{}
\end{proof}

In the next lemma we quote a special case of Lemma~2.5 in
\citeerratum{erratumharrach2013monotonicity}.
\begin{lemma}
  \label{erratum:lmm:Workhorse}
  Let $X,Y$ and $Z$ be Hilbert spaces, and let $A:X\to Y$ and $B:X\to
  Z$ be bounded linear operators.  Then, 
  \begin{equation*}
    \exists C>0:\; \|Ax\|\leq C\|Bx\| \quad \forall x\in X 
    \qquad\text{if and only if}\qquad
    \Rcal(A^*)\tm\Rcal(B^*) \,.
  \end{equation*}
\end{lemma}

Now we give the proof of Theorem~\ref{erratum:thm:LocPot1}. 

\begin{proof}[Proof of Theorem~\ref{erratum:thm:LocPot1}]
  Let $V\tm\LtSd$ be a finite dimensional subspace. 
  We denote by $P_V:\LtSd\to\LtSd$ the orthogonal projection on $V$.
  Combining Lemma~\ref{erratum:lmm:LocPot2} with a simple dimensionality
  argument (see~\citeerratum[Lmm.~4.7]{erratumharrach2017monotonicity})  shows that 
  \begin{equation*}
    Z 
    \,\not\tm\, \Rcal\bigl(
    \begin{pmatrix}
      L_{q,M}^* & L_{0,M}^*
    \end{pmatrix}\bigr) + V 
    \,=\, \Rcal(
    \begin{pmatrix}
      L_{q,M}^* & L_{0,M}^* & P_V
    \end{pmatrix}) \,,
  \end{equation*}
  where $Z\tm \Rcal(L_{q,E}^*)$ denotes the subspace in
  Lemma~\ref{erratum:lmm:LocPot2}. 
  Thus,
  \begin{equation*}
    \Rcal(L_{q,E}^*) 
    \,\not\tm\, \Rcal\bigl(
    \begin{pmatrix}
      L_{q,M}^* & L_{0,M}^*
    \end{pmatrix}\bigr) + V 
    \,=\, \Rcal(
    \begin{pmatrix}
      L_{q,M}^* & L_{0,M}^* & P_V
    \end{pmatrix}) \,,
  \end{equation*}
  and accordingly Lemma~\ref{erratum:lmm:Workhorse} implies that there is no
  constant $C>0$ such that 
  \begin{equation*}
    \begin{split}
      \|L_{q,E}g\|^2_{\Lt(E)}
      &\,\leq\, C^2\biggl\|
      \begin{pmatrix}
        L_{q,M} \\ L_{0,M} \\ P_V
      \end{pmatrix}
      g\biggr\|^2_{\Lt(M)\times\Lt(M)\times\Lt(\Sd)}\\
      &\,=\, C^2\bigl( \|L_{q,M}g\|^2_{\Lt(M)} 
      + \|L_{0,M}g\|^2_{\Lt(M)} + \|P_Vg\|^2_{\Lt(\Sd)} \bigr)
    \end{split}
  \end{equation*}
  for all $g\in\LtSd$.
  Hence, there exists as sequence $(\gtilde_m)_{m\in\N}\tm\LtSd$ such
  that 
  \begin{equation*}
    \begin{split}
      \|L_{q,E}\gtilde_m\|_{\Lt(E)}&\to\infty \,,\\
      \|L_{q,M}\gtilde_m\|_{\Lt(M)} 
      + \|L_{0,M}\gtilde_m\|_{\Lt(M)} + \|P_V\gtilde_m\|_{\LtSd}&\to 0 
      \qquad\text{as } m\to\infty\,.
    \end{split}
  \end{equation*}
  Setting $g_m:=\gtilde_m-P_V\gtilde_m\in \Vperp\tm\LtSd$ for any
  $m\in\N$, we finally obtain
  \begin{equation*}
    \|L_{q,E}g_m\|_{\Lt(E)} 
    \,\geq\,
    \|L_{q,E}\gtilde_m\|_{\Lt(E)}-\|L_{q,E}\|\|P_V\gtilde_m\|_{\LtSd}
    \,\to\,\infty \qquad
    \text{as } m\to\infty \,,
  \end{equation*}
  and
  \begin{multline*}
    \|L_{q,M}g_m\|_{\Lt(M)} + \|L_{0,M}g_m\|_{\Lt(M)}
    \,\leq\, \|L_{q,M}\gtilde_m\|_{\Lt(M)} 
    + \|L_{0,M}\gtilde_m\|_{\Lt(M)}\\
    + (\|L_{q,M}\|+\|L_{0,M}\|)\|P_V\gtilde_m\|_{\LtSd} 
    \,\to\, 0 \qquad
    \text{as } m\to\infty \,.
  \end{multline*}
  Since $L_{q,E}g_m=u_{q,g_m}|_E$, $L_{q,M}g_m=u_{q,g_m}|_M$, and 
  $L_{0,M}g_m=\ui_{g_m}|_M$, this ends the proof. 
\end{proof}

\section{Correction of the statement and of the proof of Theorem~5.3 in {[3]}}
  
\label{erratum:sec:correction}

\begin{theorem}
  \label{erratum:thm:Shape3}
  Let $B,D\tm\Rd$ be open and Lipschitz bounded such that $\di D$ is
  piecewise $C^{1,1}$ smooth, and $\R^d\setminus\ol{B}$ as well as
  $\R^d\setminus\ol{D}$ are connected.
  Let ${q\in\LinftyCRd}$ with $\supp(q)=\ol{D}$, and suppose that
  $-1<\qmin\leq q\leq\qmax<\infty$ a.e.\ on $D$ for some constants
  $\qmin,\qmax\in\R$. 

  Furthermore, we assume that for any point $x\in\di D$ on the
  boundary of $D$, there exists a connected unbounded neighborhood
  $O\tm\Rd$ of $x$ such that, for $E:=O\cap D$, 
  \begin{equation}
    \label{erratum:eq:LocalDefiniteness}
    q|_E\geq\qminE>0
    \qquad\text{or}\qquad 
    q|_E\leq\qmaxE<0 
  \end{equation}
  for some constants $\qminE,\qmaxE\in\R$.
  \begin{itemize}
  \item[(a)] If $D\tm B$, then there exists a constant $C>0$ such that
    \begin{equation*}
      \alpha T_B \,\leqfin\, \real(F_q) \,\leqfin\, \beta T_B
      \quad\text{for all } \alpha \leq \min\{0,\qmin\}\,,\; 
      \beta \geq \max\{0,C\qmax\} \,.
    \end{equation*}
  \item[(b)]  If $D\not\tm B$, then
    \begin{equation*}
      \alpha T_B \,\not\leqfin\, \real(F_q) 
      \quad\text{for any } \alpha\in\R
      \quad\text{or}\quad
      \real(F_q) \,\not\leqfin\, \beta T_B
      \quad\text{for any } \beta\in\R \,.
    \end{equation*}
  \end{itemize}
\end{theorem}

\begin{remark}
  The assumptions on $B$ and $D$ as well as the 
  \emph{local definiteness assumption} \eqref{erratum:eq:LocalDefiniteness} in
  Theorem~\ref{erratum:thm:Shape3} are stronger than in the original version
  of Theorem~5.3 in \protect\citeerratum{GriHar18}. 
  \hfill$\lozenge$
\end{remark}

\begin{proof}[Proof of Theorem~\ref{erratum:thm:Shape3}]
  If $D\tm B$, then Corollary~3.4 and
  Theorem~4.5 in \protect\citeerratum{GriHar18} with $q_1=0$ and $q_2=q$ show that
  there exists a constant $C>0$ and a finite dimensional subspace
  $V\tm\LtSd$ such that, for all $g\in\Vperp$ and any
  $\beta\geq\max\{0,C\qmax\}$, 
  \begin{equation*}
    \begin{split}
      \real\biggl(\int_\Sd g \, \ol{F_qg}\ds\biggr)
      &\,\leq\, k^2 \int_D q |u_{q,g}|^2 \dx
      \,\leq\, k^2 \qmax \int_D |u_{q,g}|^2 \dx\\
      &\,\leq\, k^2 C \qmax \int_D |\ui_{g}|^2 \dx
      \,\leq\, k^2 \beta \int_B |\ui_{g}|^2 \dx \,.
    \end{split}
  \end{equation*}

  Similarly, Theorem~3.2 in \protect\citeerratum{GriHar18} with $q_1=0$ and $q_2=q$
  shows that there exists a finite dimensional subspace $V\tm\LtSd$ 
  such that, for all $g\in\Vperp$ and any $\alpha\leq\min\{0,\qmin\}$, 
  \begin{equation*}
    \begin{split}
      \real\biggl(\int_\Sd g\, \ol{F_qg}\ds\biggr)
      &\,\geq\, k^2 \int_D q|\ui_g|^2\dx 
      \,\geq\, k^2 \qmin \int_D |\ui_g|^2\dx 
      \,\geq\, k^2 \alpha \int_B |\ui_g|^2\dx \,,
    \end{split}
  \end{equation*}
  and part (a) is proven.
  
  We prove part (b) by contradiction.  
  Since $D\not\tm B$, $U:=D\setminus B$ is not empty, and there exists 
  $x\in \ol{U}\cap\di D$ as well as a connected unbounded open
  neighborhood $O\tm\Rd$ of $x$ with $O\cap D\tm U$ and 
  $O\cap B=\emptyset$, such that \eqref{erratum:eq:LocalDefiniteness} is
  satisfied with $E:=O\cap D$. 
  Furthermore, let $R>0$ be large enough such that $B,D\tm\BR$.
  Without loss of generality we assume that $O\cap\BR$, and 
  $\BR\setminus \ol{O}$ are connected. 
  
  We first assume that $q|_E\geq\qminE>0$, and that
  $\real(F_q)\leqfin\beta T_B$ for some $\beta\in\R$.
  Using the monotonicity relation (3.1) in Theorem~3.2 of
  \protect\citeerratum{GriHar18} with $q_1=0$ and $q_2=q$, we find that there exists
  a finite dimensional subspace $V\tm\LtSd$ such that, for any
  $g\in\Vperp$, 
  \begin{equation*}
    \begin{split}
      0
      &\,\geq\, \int_\Sd g(\ol{\real(F_q)g - \beta T_B g})\ds
      \,\geq\, k^2 \int_\BR (q-\beta\chi_B) |\ui_g|^2 \dx\\
      &\,=\, k^2 \int_{\BR\setminus\ol{O}} (q-\beta\chi_B) |\ui_g|^2 \dx
      + k^2 \int_{\BR\cap O} (q-\beta\chi_B) |\ui_g|^2 \dx\\
      &\,\geq\, -k^2(\|q\|_{\Linfty(\Rd)}+|\beta|)
      \int_{\BR\setminus\ol{O}} |\ui_g|^2 \dx
      + k^2 \qminE \int_{E} |\ui_g|^2 \dx \,. 
    \end{split}
  \end{equation*}
  However, this contradicts Theorem~4.1 in \protect\citeerratum{GriHar18} with $B=E$, 
  $D=\BR\setminus\ol{O}$, and $q=0$, which guarantees the existence of
  a sequence $(g_m)_{m\in\N}\tm\Vperp$ with 
  \begin{equation*}
    \int_E|\ui_{g_m}|^2\dx \to \infty
    \qquad\text{and}\qquad
    \int_{\BR\setminus\ol{O}}|\ui_{g_m}|^2\dx \to 0 
    \qquad\text{as } m\to\infty \,.
  \end{equation*}
  Consequently, $\real(F_q)\not\leqfin\beta T_B$ for all $\beta\in\R$.

  On the other hand, if $q|_{E}\leq\qmaxE<0$, and if
  $\alpha T_B\leqfin\real(F_q)$ for some $\alpha\in\R$, then the
  monotonicity relation (3.3) in Corollary~3.4 of \protect\citeerratum{GriHar18} with
  $q_1=0$ and $q_2=q$ shows that there exists a finite dimensional
  subspace $V\tm\LtSd$ such that, for any $g\in\Vperp$, 
  \begin{equation*}
    \begin{split}
      0
      &\,\leq\, \int_\Sd g(\ol{\real(F_q)g - \alpha T_B g})\ds
      \,\leq\, k^2 \int_\BR (q |u_{q,g}|^2 
      - \alpha \chi_B|\ui_g|^2) \dx\\
      &\,=\, k^2 \int_{\BR\setminus\ol{O}} 
      (q |u_{q,g}|^2 - \alpha \chi_B|\ui_g|^2) \dx
      + k^2 \int_{\BR\cap O} 
      (q |u_{q,g}|^2 - \alpha \chi_B|\ui_g|^2) \dx\\
      &\,\leq\, k^2 \qmax \int_{\BR\setminus\ol{O}} |u_{q,g}|^2 \dx
      + k^2  |\alpha| \int_{\BR\setminus\ol{O}} |\ui_g|^2 \dx
      + k^2 \qmaxE \int_{E} |u_{q,g}|^2 \dx \,.
    \end{split}
  \end{equation*}
  Let $M:=\BR\setminus\ol{O}$.
  Since $\di D$ is piecewise $C^{1,1}$ smooth, there is a connected
  subset $\Gamma\tm\di E\setminus\ol{M}$ that is relatively open and
  $C^{1,1}$ smooth.
  Applying Theorem~\ref{erratum:thm:LocPot1} we find that there exists a sequence
  $(g_m)_{m\in\N}\tm\Vperp$ such that
  \begin{equation*}
    \int_E |u_{q,g_m}|^2\dx \to \infty 
    \qquad\text{and}\qquad
    \int_{\BR\setminus\ol{O}} |u_{q,g_m}|^2 + |\ui_{g_m}|^2\dx \to 0 \qquad 
    \text{as } m\to\infty \,.
  \end{equation*}  
  However, since $\qmaxE<0$, this gives a contradiction.
  Consequently, $\alpha T_B \not\leqfin \real(F_q)$ for all
  $\alpha\in\R$, which ends the proof of part (b).
\end{proof}


\bibliographystyleerratum{abbrv}
\bibliographyerratum{literaturliste_erratum.bib}

\end{document}